\documentclass[oneside]{amsart}
\usepackage{amssymb,amsmath,stmaryrd,txfonts,mathrsfs,amsthm}

\usepackage[all,2cell]{xy}
\usepackage[neveradjust]{paralist}
\usepackage{hyperref}
\usepackage{mathtools}
\usepackage{multirow}
\usepackage[outline]{contour}
\contourlength{1.2pt}
\usepackage{tikz}
\usepackage{tikz-cd}
\usepackage{xcolor}
\usepackage{framed,color}
\usetikzlibrary{matrix,arrows,decorations.pathmorphing,positioning}
\usetikzlibrary{intersections,decorations.markings}
\usetikzlibrary{arrows,positioning,fit,matrix,shapes.geometric,external}
\usetikzlibrary{backgrounds,circuits,circuits.ee.IEC,shapes,fit,matrix}
\makeatletter
\let\ea\expandafter

\definecolor{shadecolor}{rgb}{1,0.8,0.3}
\definecolor{myurlcolor}{rgb}{0.6,0,0}
\definecolor{mycitecolor}{rgb}{0,0,0.8}
\definecolor{myrefcolor}{rgb}{0,0,0.8}
\hypersetup{colorlinks, linkcolor=myrefcolor, citecolor=mycitecolor, urlcolor=myurlcolor}

\tikzset{->-/.style={decoration={
  markings,
  mark=at position .5 with {\arrow{>}}},postaction={decorate}}}

\def\mdef#1#2{\ea\ea\ea\gdef\ea\ea\noexpand#1\ea{\ea\ensuremath\ea{#2}}}
\def\alwaysmath#1{\ea\ea\ea\global\ea\ea\ea\let\ea\ea\csname your@#1\endcsname\csname #1\endcsname
  \ea\def\csname #1\endcsname{\ensuremath{\csname your@#1\endcsname}}}

\newcommand{\define}[1]{{\bf \boldmath{#1}}}

\newcommand{\lA}{\ensuremath{\mathbb{A}}}
\newcommand{\lB}{\ensuremath{\mathbb{B}}}
\newcommand{\lC}{\ensuremath{\mathbb{C}}}
\newcommand{\lD}{\ensuremath{\mathbb{D}}}
\newcommand{\lE}{\ensuremath{\mathbb{E}}}

\newcommand{\fa}{\ensuremath{\mathfrak{a}}}

\newcommand{\fs}{\ensuremath{\mathfrak{s}}}

\newcommand{\fu}{\ensuremath{\mathfrak{u}}}

\newcommand{\fx}{\ensuremath{\mathfrak{x}}}

\mdef\fahat{\hat{\fa}}

\newcommand{\fhat}{\ensuremath{\hat{f}}}

\DeclareSymbolFont{bbold}{U}{bbold}{m}{n}
\DeclareSymbolFontAlphabet{\mathbbb}{bbold}

\mdef\del{\partial}
\mdef\delbar{\overline{\partial}}

\mdef\hf{\textstyle\frac{1}{2}}
\mdef\thrd{\textstyle\frac{1}{3}}
\mdef\qtr{\textstyle\frac{1}{4}}

\newcommand{\id}{\mathm{id}}

\SelectTips{cm}{}
\newdir{ >}{{}*!/-10pt/@{>}}    

\mdef\Id{\mathrm{Id}}
\mdef\id{\mathrm{id}}
\alwaysmath{ell}
\alwaysmath{infty}
\alwaysmath{odot}
\def\frc#1/#2.{\frac{#1}{#2}}   
\mdef\ten{\mathrel{\otimes}}
\mdef\bigten{\bigotimes}
\mdef\sqten{\mathrel{\boxtimes}}
\def\pow(#1,#2){\mathop{\pitchfork}(#1,#2)} 

\newcommand{\too}[1][]{\ensuremath{\overset{#1}{\longrightarrow}}}

\mdef\we{\overset{\sim}{\longrightarrow}}
\mdef\leftwe{\overset{\sim}{\longleftarrow}}

\let\maps\colon

\def\rightarrowtailfill@{\arrowfill@{\Yright\joinrel\relbar}\relbar\rightarrow}
\newcommand\xrightarrowtail[2][]{\ext@arrow 0055{\rightarrowtailfill@}{#1}{#2}}

\def\twoheadrightarrowfill@{\arrowfill@{\relbar\joinrel\relbar}\relbar\twoheadrightarrow}
\newcommand\xtwoheadrightarrow[2][]{\ext@arrow 0055{\twoheadrightarrowfill@}{#1}{#2}}

\def\slashedarrowfill@#1#2#3#4#5{%
  $\m@th\thickmuskip0mu\medmuskip\thickmuskip\thinmuskip\thickmuskip
   \relax#5#1\mkern-7mu%
   \cleaders\hbox{$#5\mkern-2mu#2\mkern-2mu$}\hfill
   \mathclap{#3}\mathclap{#2}%
   \cleaders\hbox{$#5\mkern-2mu#2\mkern-2mu$}\hfill
   \mkern-7mu#4$%
}
\def\rightslashedarrowfill@{%
  \slashedarrowfill@\relbar\relbar\mapstochar\rightarrow}
\newcommand\xslashedrightarrow[2][]{%
  \ext@arrow 0055{\rightslashedarrowfill@}{#1}{#2}}
\mdef\hto{\xslashedrightarrow{}}
\mdef\htoo{\xslashedrightarrow{\quad}}

\def\defthm#1#2{%
  \newtheorem{#1}{#2}[section]%
  \expandafter\def\csname #1autorefname\endcsname{#2}%
  \expandafter\let\csname c@#1\endcsname\c@thm}
\newtheorem{thm}{Theorem}[section]

\defthm{cor}{Corollary}
\defthm{prop}{Proposition}
\defthm{lem}{Lemma}
\defthm{sch}{Scholium}
\defthm{assume}{Assumption}
\defthm{claim}{Claim}
\defthm{conj}{Conjecture}
\defthm{hyp}{Hypothesis}
\defthm{fact}{Fact}
\theoremstyle{definition}
\defthm{defn}{Definition}
\defthm{notn}{Notation}
\theoremstyle{remark}
\defthm{rmk}{Remark}
\defthm{eg}{Example}
\defthm{egs}{Examples}
\defthm{ex}{Exercise}
\defthm{ceg}{Counterexample}

\def\thmqedhere{\expandafter\csname\csname @currenvir\endcsname @qed\endcsname}

\let\c@equation\c@thm
\numberwithin{equation}{section}

\mathtoolsset{showonlyrefs,showmanualtags}

\newlength\oldleftmargini       
\newlength\oldleftmarginii
\newlength\oldleftmarginiii
\newlength\oldleftmarginiv
\newlength\oldleftmarginv
\newlength\oldleftmarginvi
\newcount\maxenum
\newif\ifkillspacing
\def\@adjust@enum@labelwidth{%
  \advance\@listdepth by 1\relax
  \ifkillspacing                
    \csname c@\@enumctr\endcsname\maxenum
    \settowidth{\@tempdima}{%
      \csname label\@enumctr\endcsname\hspace{\labelsep}}%
    \csname leftmargin\romannumeral\@listdepth\endcsname
      \@tempdima
  \else                         
    \csname fixspacing\romannumeral\@listdepth\endcsname
  \fi
  \advance\@listdepth by -1\relax}
\def\fixspacingi{\ifnum\oldleftmargini=0\setlength\oldleftmargini\leftmargini\else\setlength\leftmargini\oldleftmargini\fi}
\def\fixspacingii{\ifnum\oldleftmarginii=0\setlength\oldleftmarginii\leftmarginii\else\setlength\leftmarginii\oldleftmarginii\fi}
\def\fixspacingiii{\ifnum\oldleftmarginiii=0\setlength\oldleftmarginiii\leftmarginiii\else\setlength\leftmarginiii\oldleftmarginiii\fi}
\def\fixspacingiv{\ifnum\oldleftmarginiv=0\setlength\oldleftmarginiv\leftmarginiv\else\setlength\leftmarginiv\oldleftmarginiv\fi}
\def\fixspacingv{\ifnum\oldleftmarginv=0\setlength\oldleftmarginv\leftmarginv\else\setlength\leftmarginv\oldleftmarginv\fi}
\def\fixspacingvi{\ifnum\oldleftmarginvi=0\setlength\oldleftmarginvi\leftmarginvi\else\setlength\leftmarginvi\oldleftmarginvi\fi}

\def\pl@label#1#2{%
  \edef\pl@the{\noexpand#1{\@enumctr}}%
  \pl@lab\expandafter{\the\pl@lab\csname yourthe\@enumctr\endcsname}%
  \advance\@tempcnta1
  \pl@loop}
\def\@enumlabel@#1[#2]{%
  \@plmylabeltrue
  \@tempcnta0
  \pl@lab{}%
  \let\pl@the\pl@qmark
  \expandafter\pl@loop\@gobble#2\@@@
  \ifnum\@tempcnta=1\else
    \PackageWarning{paralist}{Incorrect label; no or multiple
      counters.\MessageBreak The label is: \@gobble#2}%
  \fi
  \expandafter\edef\csname label\@enumctr\endcsname{\the\pl@lab}%
  \expandafter\edef\csname the\@enumctr\endcsname{\the\pl@lab}%
  \expandafter\let\csname yourthe\@enumctr\endcsname\pl@the
  #1}

\alwaysmath{alpha}
\alwaysmath{beta}
\alwaysmath{gamma}
\alwaysmath{Gamma}
\alwaysmath{delta}
\alwaysmath{Delta}
\alwaysmath{epsilon}
\mdef\ep{\varepsilon}
\alwaysmath{zeta}
\alwaysmath{eta}
\alwaysmath{theta}
\alwaysmath{Theta}
\alwaysmath{iota}
\alwaysmath{kappa}
\alwaysmath{lambda}
\alwaysmath{Lambda}
\alwaysmath{mu}
\alwaysmath{nu}
\alwaysmath{xi}
\alwaysmath{pi}
\alwaysmath{rho}
\alwaysmath{sigma}
\alwaysmath{Sigma}
\alwaysmath{tau}
\alwaysmath{upsilon}
\alwaysmath{Upsilon}
\alwaysmath{phi}
\alwaysmath{Pi}
\alwaysmath{Phi}
\mdef\ph{\varphi}
\alwaysmath{chi}
\alwaysmath{psi}
\alwaysmath{Psi}
\alwaysmath{omega}
\alwaysmath{Omega}

\makeatother

\UseAllTwocells
\title{A bicategory of decorated cospans}
\author{Kenny Courser}
\mdef\cMod{\mathcal{M}\mathit{od}}
\mdef\cCat{\mathcal{C}\mathit{at}}
\mdef\cTwocat{2\text{-}\mathcal{C}\mathit{at}}
\mdef\cBicat{\mathcal{B}\mathit{icat}}
\mdef\lMod{\mathbb{M}\mathsf{od}}
\mdef\lnCob{n\mathbb{C}\mathsf{ob}}
\mdef\lProf{\mathbb{P}\mathsf{rof}}
\mdef\cDbl{\mathcal{D}\mathit{bl}}
\mdef\fchk{\check{f}}
\mdef\conj{\Yleft}
\mdef\Conj{\mathcal{C}\mathit{onj}}

\begin{document}
\maketitle
\begin{center}   
  {\small Department of Mathematics  \\
    University of California  \\
  Riverside CA, USA 92521  \\ }
  \vspace{0.3cm}   
  {\small email:  courser@math.ucr.edu\\} 
  \vspace{0.3cm}   
  {\small \today}
  \vspace{0.3cm}   
\end{center}   

\begin{abstract}
\noindent
If $\mathbf{C}$ is a category with pullbacks then there is a bicategory with the same objects as $\mathbf{C}$, spans as morphisms, and maps of spans as 2-morphisms, as shown by Benabou. Fong has developed a theory of `decorated cospans', which are cospans in $\mathbf{C}$ equipped with extra structure. This extra structure arises from a symmetric lax monoidal functor $F \colon \mathbf{C} \to \mathbf{D}$; we use this functor to `decorate' each cospan with apex $N \in \mathbf{C}$ with an element of $F(N)$. Using a result of Shulman, we show that when $\mathbf{C}$ has finite colimits, decorated cospans are morphisms in a symmetric monoidal bicategory. We illustrate our construction with examples from electrical engineering and the theory of chemical reaction networks.
\end{abstract}

\section{Introduction}
Networks are becoming increasingly important in applied mathematics and engineering, and developing a general theory of networks will require new ideas connecting these subjects to category theory.  We think of a network with some inputs $X$ and outputs $Y$ as a cospan $X \rightarrow N \leftarrow Y$ in some category $\bold{C}$, and compose these cospans using pushouts.  Typically, however, the apex $N$ is equipped with some extra structure, so we also need a way to compose the extra structures. This was recently developed by Fong  \cite{Fon}, who gave a general recipe for constructing `decorated cospan categories' and functors between these. Baez, Fong and Pollard have used decorated cospans to prove new results about electrical circuits and Markov processes \cite{BF,BFP,Pol}.

However, besides asking whether two networks are equal, it makes sense to ask if they are isomorphic.   Thus, cospans are not merely morphisms in a category, but morphisms in a bicategory---indeed, this example appeared already in Benabou's original paper on bicategories \cite{Be}.   In fact, if $\bold{C}$ is a category with finite colimits, Stay has proved there is a symmetric monoidal bicategory whose morphisms are cospans in $\bold{C}$ \cite{Stay}.  For applications to network theory, we need to generalize this result to decorated cospans.  Stay's result, which allows $\bold{C}$ to be a 2-category, used Hoffnung's work on tricategories \cite{Hoff}, but for our purposes an easier approach is to use Shulman's technique for constructing symmetric monoidal bicategories \cite{Shul}.  This involves first constructing a symmetric monoidal pseudo double category.

Another work aimed towards the application of double categories is that of Lerman and Spivak \cite{Lerm}. Following Brockett \cite{Bro}, they model an open dynamical system as a smooth map between manifolds $F \maps Q \to TM$ which when composed with the projection $\pi_{M} \colon TM \to M$ gives a surjective submersion $p \maps Q \to M$. When $Q = M$ and this surjective submersion is the identity, an open dynamical system reduces to a smooth vector field on $M$, which is an ordinary dynamical system. They construct a monoidal double category $\bold{SSub^\square}$ whose objects are surjective submersions, and a lax monoidal double functor from $\bold{SSub^\square}$ to the monidal double category $\bold{RelVect^\square}$ of vector spaces, linear maps, linear relations and inclusions of relations assigning to each surjective submersion $p \maps Q \to M$ the vector space of open dynamical systems having this as their underlying surjective submersion. The horizontal morphisms in $\bold{SSub^\square}$ are `dynamical morphisms', which give maps of open dynamical systems; the vertical morphisms in $\bold{SSub^\square}$ are `interconnection morphisms', which can be used to describe the process of connecting open dynamical systems and the 2-morphisms are certain commuting squares. Lerman and Spivak recover one of the main results of a previous work by DeVille and Lerman \cite{DeLerm} in which a `fibration of networks of manifolds' gives a 2-morphism in $\bold{RelVect^\square}$. This is just a special case of the above result in which dynamical morphisms give rise to maps between certain families of open systems. In both these works, open systems are described using \emph{objects}. This differs from the framework discussed in the current paper, in which open systems are described as \emph{morphisms}, while objects describe inputs and outputs of those systems \cite{BP}.

If $\bold{C}$ is a category with finite colimits, chosen pushouts and binary coproducts and $F \colon (\bold{C},+) \to (\bold{D},\otimes)$ is a symmetric lax monoidal functor, then we can decorate the apex of a cospan in $\bold{C}$, which is an object $N \in \bold{C}$, with an element of $F(N) \in \bold{D}$ given by a morphism $f \colon I \to F(N)$ where $I$ is the unit object for the tensor product in $\bold{D}$. As a specific example, let $\bold{C} = \bold{FinSet}$ and $\bold{D} = \bold{Set}$ and let $F \colon\bold{Finset} \to \bold{Set}$ be the functor that assigns to each finite set $N$ the set $F(N)$ of all ways of assigning `weights' given by positive real numbers to edges of a graph whose vertex set is $N$. To see what this would look like, let $N$ be an arbitrary 3 element set. Then one possible assignment of weighted edges to $N$ would be:

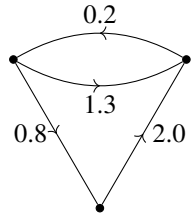
\begin{center}
  \begin{tikzpicture}[auto,scale=2.3]
    \node[circle,draw,inner sep=1pt,fill]         (A) at (0,0) {};
    \node[circle,draw,inner sep=1pt,fill]         (B) at (1,0) {};
    \node[circle,draw,inner sep=1pt,fill]         (C) at (0.5,-.86) {};
    \path (B) edge  [bend right,->-] node[above] {0.2} (A);
    \path (A) edge  [bend right,->-] node[below] {1.3} (B);
    \path (A) edge  [->-] node[left] {0.8} (C);
    \path (C) edge   [->-] node[right] {2.0} (B);
  \end{tikzpicture}
\end{center}
In applications to electrical circuits, one could use a weighted graph of this type to represent an electrical circuit made of resistors where the weights are resistances. Here, our weights are elements of the set $L=(0,\infty)$. The above diagram is an instance of a `weighted graph', and is one possible example of an element of $F(N)$. 

From this graph, we can select subsets of nodes $X$ and $Y$ to be the inputs and outputs, respectively, which then yield maps $X \to N$ and $Y \to N$. These together with the specified element of $F(N)$ above gives us a `decorated cospan' which is a cospan $X \to N \leftarrow Y$ in $\bold{C}$ together with a map $s \colon 1 \to F(N)$, where the map $s$ specifies the decoration on the finite set $N$ by selecting an element out of $F(N)$. In our example of weighted graphs, a decorated cospan might look like:

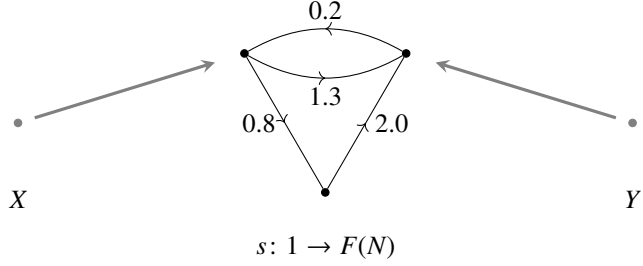
\begin{center}
  \begin{tikzpicture}[auto,scale=2.15]
    \node[circle,draw,inner sep=1pt,fill=gray,color=gray]         (x) at (-1.4,-.43) {};
    \node at (-1.4,-.9) {$X$};
    \node[circle,draw,inner sep=1pt,fill]         (A) at (0,0) {};
    \node[circle,draw,inner sep=1pt,fill]         (B) at (1,0) {};
    \node[circle,draw,inner sep=1pt,fill]         (C) at (0.5,-.86) {};
    \node[circle,draw,inner sep=1pt,fill=gray,color=gray]         (y1) at (2.4,-.43) {};
    \node at (2.4,-.9) {$Y$};
    \path (B) edge  [bend right,->-] node[above] {0.2} (A);
    \path (A) edge  [bend right,->-] node[below] {1.3} (B);
    \path (A) edge  [->-] node[left] {0.8} (C);
    \path (C) edge  [->-] node[right] {2.0} (B);
    \path[color=gray, very thick, shorten >=10pt, shorten <=5pt, ->, >=stealth] (x) edge (A);
    \path[color=gray, very thick, shorten >=10pt, shorten <=5pt, ->, >=stealth] (y1) edge (B);
    \node at (0.5,-1.2) {$s \colon 1 \to F(N)$};
  \end{tikzpicture}
\end{center}

In the application to electrical circuits, the maps $X \to N$ and $Y \to N$ specify the inputs and outputs of the circuit, respectively, and are not required to be injective maps. In Section \ref{Section 5}, we present an example where not all of the maps are injections. We can then compose these electrical circuits by identifying the inputs of one with the outputs of another. We refer the curious reader to Fong \cite{Fon}.

\section{Overview}

Throughout this paper, whenever we say $\bold{C}$ is a category with finite colimits, we mean a category $\bold{C}$ with finite colimits and with chosen pushouts and coproducts for every pair of objects in $\bold{C}$. Fong's main result on decorated cospans \cite{Fon} is the following theorem:

\begin{thm}
Let $(\bold{C},+)$ be a category with finite colimits and let $(\bold{D}, \otimes)$ be a symmetric monoidal category. Let $F \colon \bold{C} \to \bold{D}$ be a symmetric lax monoidal functor. Then FCospan($\bold{C}$) is a symmetric monoidal category, where FCospan($\bold{C}$) is the category whose objects are that of $\bold{C}$ and whose morphisms are given by isomorphism classes of F-decorated cospans, where an F-decorated cospan is a pair

\[
    \left(
    \begin{aligned}
      \xymatrix{
	& N \\  
	X \ar[ur]^{i} && Y \ar[ul]_{o}
      }
    \end{aligned}
    ,
    \qquad
    \begin{aligned}
      \xymatrix{
	F(N) \\
	I \ar[u]_{s}
      }
    \end{aligned}
    \right)
\]

\noindent  and the composite of this $F$-decorated cospan with 
\[
\left(
\begin{aligned}
\xymatrix{
& N^\prime \\
 Y \ar[ur]^{i^\prime} && Z \ar[ul]_{o^\prime}
}
\end{aligned}
,
\qquad
\begin{aligned}
\xymatrix{
F(N^\prime) \\
I \ar[u]_{s^\prime}
}
\end{aligned}
\right)
\]
is given by
\[
\left(
\begin{aligned}
\xymatrix{
& N +_Y N^\prime \\
 X \ar[ur]^{J_{N} \circ i} && Z \ar[ul]_{J_{N^\prime} \circ o^\prime}
}
\end{aligned}
,
\qquad
\begin{aligned}
\xymatrix{
F(N+_Y N^\prime) \\
I \ar[u]_{s^{\prime \prime}}
}
\end{aligned}
\right)
\] 
where $s^{\prime \prime}$ is the composite
\[I \xrightarrow{\lambda^{-1}} I \otimes I \xrightarrow{s \otimes s^\prime} F(N) \otimes F(N^\prime) \xrightarrow{\phi_{N,N^\prime}} F(N + N^\prime) \xrightarrow{F(J_{N,N^\prime})} F(N+_{Y}N^\prime).
\]
Here, $\phi_{N,N^\prime}$ is the natural transformation of the lax monoidal functor $F$, $J_{N,N^\prime} \colon N + N^\prime \to N +_Y N^\prime$ is the natural morphism from the coproduct to the pushout, and the maps $J_{N}$ and $J_{N^\prime}$ are the map $J_{N,N^\prime}$ restricted to $N$ and $N^\prime$, respectively.
\end{thm}

Fong's result is actually slightly more general than this in that he only requires $(\bold{D}, \otimes)$ to be braided monoidal. But actually, without loss of generality, we may assume that $(\bold{D}, \otimes) = (\bold{Set}, \times)$ due to the existence of the \textbf{global sections functor}, which is the braided lax monoidal functor $G \colon \bold{D} \to \bold{Set}$ defined by $d \mapsto \textnormal{hom}(I,d)$, where $I$ is the unit object for the tensor product of $(\bold{D}, \otimes)$.

In this paper, we will not take isomorphism classes, and thus Fong's result can be viewed as a decategorification of the result in this paper. Our goal is to prove the following theorem:

\begin{thm}\label{Theorem 2.2}
Let $(\bold{C},+)$ be a category with finite colimits and let $(\bold{D}, \otimes)$ be a symmetric monoidal category. Let $F \colon \bold{C} \to \bold{D}$ be a symmetric lax monoidal functor. Then $\text{FCospan}(\bold{C})$ is a symmetric monoidal bicategory, where $\text{FCospan}(\bold{C})$ is the category whose objects are that of $\bold{C}$, whose morphisms are given by $F$-decorated cospans and whose 2-morphisms are given by globular maps of $F$-decorated cospans, where a globular map of $F$-decorated cospans is a pair of commuting diagrams of the following form.

 \[
    \left(
    \begin{aligned}
      \xymatrix{
	& N \ar[dd]^{h} \\  
	X \ar[ur]^{i} \ar[dr]_{i^\prime} && Y \ar[ul]_{o} \ar[dl]^{o^\prime}\\
          & N'
      }
    \end{aligned}
    ,
    \qquad
    \begin{aligned}
      \xymatrix{
	& F(N) \ar[dd]^{F(h)} \\
	I \ar[ur]^{s_{1}} \ar[dr]_{s_{2}} \\
           & F(N^\prime)
      }
    \end{aligned}
    \right)
  \]

\end{thm}

\begin{proof}
Let $\bold{C}$ be a category with finite colimits and let $\bold{D}$ be a symmetric monoidal category. Note that $\bold{C}$ then becomes symmetric monoidal with the coproduct as the tensor product and the initial object, which we denote as $0$, for the unit. Let $F \colon (\bold{C},+) \to (\bold{D},\otimes)$ be a symmetric lax monoidal functor. By Proposition \ref{Proposition 4.1}, Theorem \ref{Proposition 4.2} and the definition of `fibrant', we have that Cospan($\bold{C}$) is a fibrant symmetric monoidal pseudo double category. Denote by $\bold{BD}$ the delooping of the symmetric monoidal category $\bold{D}$ into a one-object bicategory. The one object bicategory $\bold{BD}$ can then be viewed as a symmetric monoidal pseudo double category as proven in Proposition \ref{Proposition 4.3} and Proposition \ref{Proposition 4.4}. In Proposition \ref{Proposition 4.5}, we construct a symmetric lax monoidal double functor $F^\prime \colon\text{Cospan}(\bold{C}) \to \bold{BD}$ such that $F^\prime$ acts as $F$ on vertical 1-morphisms and horizontal 1-cells, which are morphisms and cospans in $\bold{C}$, respectively. Viewing the trivial category $\bold{1}$ as a symmetric monoidal pseudo double category, define a symmetric oplax monoidal double functor $E \colon\bold{1} \to \bold{BD}$ where $E$ picks out the unit object of $\bold{BD}$. We then construct in Theorem \ref{Theorem 4.8} the pseudo comma double category $(E/F^\prime)$, and show that this is the symmetric monoidal pseudo double category of $F$-decorated cospans in $\bold{C}$.  We show this symmetric monoidal pseudo double category is fibrant in Proposition \ref{Proposition 4.9}, and applying the following result of Shulman \cite{Shul} yields $F$Cospan$(\bold{C})$ as the `horizontal bicategory' of $(E/F^\prime)$, denoted as $H(E/F^\prime)$. This completes the proof of the theorem.
\end{proof}

\begin{thm}[Shulman]
Let $\lD$ be a fibrant symmetric monoidal pseudo double category. Then $H(\bold{\lD})$ is a symmetric monoidal bicategory, where $H(\bold{\lD})$ is the horizontal bicategory of $\lD$.
\end{thm}

This then gives a symmetric monoidal bicategory whose objects are objects of $\bold{C}$, whose morphisms are decorated cospans, and whose 2-morphisms are pairs of commuting diagrams as above. In what follows we denote a double category as $\lD$, using font as such, and regular categories as well as bicategories as $\bold{D}$.

\section{Definitions and background}

Pseudo double categories, also known as weak double categories, have been studied by Fiore \cite{Fiore} and Pare and Grandis \cite{Gran}. Before formally defining them, it is helpful to have the following picture in mind. A pseudo double category has 2-morphisms shaped like:

\begin{equation}\label{eq:square}
  \xymatrix@-.5pc{
    A \ar[r]|{|}^{M}  \ar[d]_f \ar@{}[dr]|{\Downarrow a}&
    B\ar[d]^g\\
    C \ar[r]|{|}_N & D
  }
\end{equation}

We call $A, B, C$ and $D$ \textbf{objects} or \textbf{0-cells}, $f$ and $g$ \textbf{vertical 1-morphisms}, $M$ and $N$ \textbf{horizontal 1-cells} and $a$ a \textbf{2-morphism}. Note that a vertical 1-morphism is a morphism between 0-cells and a 2-morphism is a morphism between horizontal 1-cells. We will denote both kinds of morphisms and horizontal 1-cells as a single arrow, namely `$\to$', unless in a diagram, in which case they will be denoted as above.

We follow the notation of Shulman \cite{Shul} with the following definitions.

\begin{defn}
A \textbf{pseudo double category} $\lD$, or $\textbf{double category}$ for short, consists of a category of objects $\bold{D_{0}}$ and a category of arrows $\bold{D_{1}}$ with the following functors
\begin{center}
$U\colon \bold{D_{0}} \to \bold{D_{1}}$\\
$S,T \colon \bold{D_{1}} \rightrightarrows \bold{D_{0}}$\\
$\odot \colon \bold{D_{1}} \times_{\bold{D_{0}}} \bold{D_{1}} \to \bold{D_{1}}$ (where the pullback is taken over $\bold{D_{1}} \xrightarrow[]{T} \bold{D_{0}} \xleftarrow[]{S} \bold{D_{1}}$) \\
\end{center}
 such that \\
\begin{center}
$S(U_{A})=A=T(U_{A})$\\
$S(M \odot N)=SN$\\
$T(M \odot N)=TM$\\
\end{center}
equipped with natural isomorphisms
\begin{center}

$\alpha \colon (M \odot N) \odot P \xrightarrow{\sim} M \odot (N \odot P)$\\
$\lambda \colon U_{B} \odot M \xrightarrow{\sim} M$\\
$\rho \colon M \odot U_{A} \xrightarrow{\sim} M$

\end{center}
such that $S(\alpha), S(\lambda), S(\rho), T(\alpha), T(\lambda)$ and $T(\rho)$ are all identities and that the coherence axioms of a monoidal category are satisfied. Following the notation of Shulman, objects of $\bold{D_{0}}$ are called $\textbf{0-cells}$ and morphisms of $\bold{D_{0}}$ are called $\textbf{vertical 1-morphisms}$. Objects of $\bold{D_{1}}$ are called $\textbf{horizontal 1-cells}$ and morphisms of $\bold{D_{1}}$ are called $\textbf{2-morphisms}$. The morphisms of $\bold{D_{0}}$, which are vertical 1-morphisms, will be denoted $f \colon A \to C$ and we denote a 1-cell $M$ with $S(M)=A,T(M)=B$ by $M \colon A \hto B$. Then a 2-morphism $a \colon M \to N$ of $\bold{D_{1}}$ with $S(a)=f,T(a)=g$ would look like:
\begin{equation}\label{eq:square}
  \xymatrix@-.5pc{
    A \ar[r]|{|}^{M}  \ar[d]_f \ar@{}[dr]|{\Downarrow a}&
    B\ar[d]^g\\
    C \ar[r]|{|}_N & D
  }
\end{equation}
\end{defn}

The key difference between a `strict' double category and a pseudo double category is that in a pseudo double category, horizontal composition is associative and unital only up to natural isomorphism. Equivalently, as a double category can be viewed as a category internal to $\bold{Cat}$, we can view a pseudo double category as a category `weakly' internal to $\bold{Cat}$. We will sometimes omit the word pseudo and simply say double category.

\begin{defn}
A 2-morphism where $f$ and $g$ are identities is called a \textbf{globular 2-morphism}.
\end{defn}

\begin{defn}
Let $\lD$ be a pseudo double category. Then the $\textbf{horizontal bicategory}$ of $\lD$, which we denote as $H(\lD)$, is the bicategory consisting of objects of $\lD$, 1-morphisms that are horizontal 1-cells of $\lD$ and 2-morphisms that are globular 2-morphisms.
\end{defn}

\begin{defn}
  A \textbf{monoidal double category} is a double category equipped the following
structure.
\begin{enumerate}
\item $\bold{D_{0}}$ and $\bold{D_{1}}$ are both monoidal categories.
\item If $I$ is the monoidal unit of $\bold{D_{0}}$, then $U_I$ is the
  monoidal unit of $\bold{D_{1}}$.
\item The functors $S$ and $T$ are strict monoidal, i.e.\ $S(M\ten N)
  = SM\ten SN$ and $T(M\ten N)=TM\ten TN$ and $S$ and $T$ also
  preserve the associativity and unit constraints.
\item We have globular isomorphisms
  \[\fx\maps (M_1\ten N_1)\odot (M_2\ten N_2)\too[\sim] (M_1\odot M_2)\ten (N_1\odot N_2)\]
  and
  \[\fu\maps U_{A\ten B} \too[\sim] (U_A \ten U_B)\]
  such that the following diagrams commute:
  \[\xymatrix{
    ((M_1\ten N_1)\odot (M_2\ten N_2)) \odot (M_3\ten N_3) \ar[r]^{\fx \odot 1} \ar[d]_{\alpha}
    & ((M_1\odot M_2)\ten (N_1\odot N_2)) \odot (M_3\ten N_3) \ar[d]^{\fx}\\
    (M_1\ten N_1)\odot ((M_2\ten N_2) \odot (M_3\ten N_3)) \ar[d]_{1 \odot \fx} &
    ((M_1\odot M_2)\odot M_3) \ten ((N_1\odot N_2)\odot N_3) \ar[d]^{\alpha \otimes \alpha}\\
    (M_1\ten N_1) \odot ((M_2\odot M_3) \ten (N_2\odot N_3))\ar[r]^{\fx} &
    (M_1\odot (M_2\odot M_3)) \ten (N_1\odot (N_2\odot N_3))}\]
  \[\xymatrix{(M\ten N) \odot U_{C\ten D} \ar[r]^{1 \odot \fu} \ar[d]_{\rho} &
    (M\ten N)\odot (U_C\ten U_D) \ar[d]^{\fx}\\
    M\ten N\ar@{<-}[r]^{\rho \otimes \rho} & (M\odot U_C) \ten (N\odot U_D)}\]
  \[\xymatrix{U_{A\ten B}\odot (M\ten N)  \ar[r]^{\fu \odot 1} \ar[d]_{\lambda} &
    (U_A\ten U_B)\odot (M\ten N) \ar[d]^{\fx}\\
    M\ten N\ar@{<-}[r]^{\lambda \otimes \lambda} & (U_A \odot M) \ten (U_B\odot N)}\]
\item The following diagrams commute, expressing that the
  associativity isomorphism for $\ten$ is a transformation of double
  categories.
  \[\xymatrix{
    ((M_1\ten N_1)\ten P_1) \odot ((M_2\ten N_2)\ten P_2) \ar[r]^{a \odot a}\ar[d]_{\fx} &
    (M_1\ten (N_1\ten P_1)) \odot (M_2\ten (N_2\ten P_2)) \ar[d]^{\fx}\\
    ((M_1\ten N_1) \odot (M_2\ten N_2)) \ten (P_1\odot P_2) \ar[d]_{\fx \otimes 1} &
    (M_1\odot M_2) \ten ((N_1\ten P_1)\odot (N_2\ten P_2))\ar[d]^{1 \otimes \fx} \\
    ((M_1\odot M_2) \ten(N_1\odot N_2)) \ten (P_1\odot P_2) \ar[r]^{a} &
    (M_1\odot M_2) \ten ((N_1\odot N_2)\ten (P_1\odot P_2))}\]
  \[\xymatrix{
    U_{(A\ten B)\ten C} \ar[r]^{U_{a}} \ar[d]_{\fu} & U_{A\ten (B\ten C)} \ar[d]^{\fu}\\
    U_{A\ten B} \ten U_C \ar[d]_{\fu \otimes 1} & U_A\ten U_{B\ten C}\ar[d]^{1 \otimes \fu}\\
    (U_A\ten U_B)\ten U_C \ar[r]^{a} & U_A\ten (U_B\ten U_C) }\]
\item The following diagrams commute, expressing that the unit
  isomorphisms for $\ten$ are transformations of double categories.
  \[\vcenter{\xymatrix{
      (M\ten U_I)\odot (N\ten U_I)\ar[r]^{\fx}\ar[d]_{r \odot r} &
      (M\odot N)\ten (U_I \odot U_I) \ar[d]^{1 \otimes \rho}\\
      M\odot N \ar@{<-}[r]^{r} &
      (M\odot N)\ten U_I }}\]
  \[\vcenter{\xymatrix{U_{A\ten I} \ar[r]^{\fu} \ar[dr]_{U_{r}} & U_A\ten U_I \ar[d]^{r}\\
       & U_A}}\]
  \[\vcenter{\xymatrix{
      (U_I\ten M)\odot (U_I\ten N)\ar[r]^{\fx} \ar[d]_{\ell \odot \ell} &
      (U_I \odot U_I) \ten (M\odot N) \ar[d]^{\lambda \otimes 1}\\
      M\odot N \ar@{<-}[r]^{\ell} &
      U_I\ten (M\odot N) }}\]
  \[\vcenter{\xymatrix{U_{I\ten A} \ar[r]^{\fu}\ar[dr]_{U_{\ell}} & U_I\ten U_A \ar[d]^{\ell}\\
      & U_A}}\]
  \newcounter{mondbl}
  \setcounter{mondbl}{\value{enumi}}
\end{enumerate}
Similarly, a braided monoidal double category is a monoidal double
category with the following additional structure.
\begin{enumerate}\setcounter{enumi}{\value{mondbl}}
\item $\bold{D_{0}}$ and $\bold{D_{1}}$ are braided monoidal categories.
\item The functors $S$ and $T$ are strict braided monoidal (i.e.\ they
  preserve the braidings).
\item The following diagrams commute, expressing that the braiding is
  a transformation of double categories.
  \[\xymatrix{(M_1\odot M_2)\ten (N_1\odot N_2) \ar[r]^\fs\ar[d]_\fx &
    (N_1\odot N_2)\ten (M_1 \odot M_2)\ar[d]^\fx\\
    (M_1\ten N_1)\odot (M_2\ten N_2) \ar[r]_{\fs\odot \fs} &
    (N_1\ten M_1) \odot (N_2 \ten M_2)}
  \]
  \[\xymatrix{U_A \ten U_B  \ar[d]_\fs &
    U_{A\ten B} \ar[l]_(0.40)\fu \ar[d]^{U_\fs}\\
    U_B\ten U_A &
    U_{B\ten A} \ar[l]_(0.40)\fu}
  \]
  \setcounter{mondbl}{\value{enumi}}
\end{enumerate}
Finally, a symmetric monoidal double category is a braided one such that
\begin{enumerate}\setcounter{enumi}{\value{mondbl}}
\item $\bold{D_{0}}$ and $\bold{D_{1}}$ are in fact symmetric monoidal.
\end{enumerate}
\end{defn}

\begin{defn}\label{def:companion}
  Let \lD\ be a double category and $f\maps A\to B$ a vertical
  1-morphism.  A \textbf{companion} of $f$ is a horizontal 1-cell
  $\fhat\maps A\hto B$ together with 2-morphisms
  \begin{equation*}
    \begin{array}{c}
      \xymatrix@-.5pc{
        \ar[r]|-@{|}^-{\fhat} \ar[d]_f \ar@{}[dr]|\Downarrow
        & \ar@{=}[d]\\
        \ar[r]|-@{|}_-{U_B} & }
    \end{array}\quad\text{and}\quad
    \begin{array}{c}
      \xymatrix@-.5pc{
        \ar[r]|-@{|}^-{U_A} \ar@{=}[d] \ar@{}[dr]|\Downarrow
        & \ar[d]^f\\
        \ar[r]|-@{|}_-{\fhat} & }
    \end{array}
  \end{equation*}
  such that the following equations hold.
  \begin{align}\label{eq:compeqn}
    \begin{array}{c}
      \xymatrix@-.5pc{
        \ar[r]|-@{|}^-{U_A} \ar@{=}[d] \ar@{}[dr]|\Downarrow
        & \ar[d]^f\\
        \ar[r]|-{\fhat} \ar[d]_f \ar@{}[dr]|\Downarrow
        & \ar@{=}[d]\\
        \ar[r]|-@{|}_-{U_B} & }
    \end{array} &= 
    \begin{array}{c}
      \xymatrix@-.5pc{ \ar[r]|-@{|}^-{U_A} \ar[d]_f
        \ar@{}[dr]|{\Downarrow U_f} &  \ar[d]^f\\
        \ar[r]|-@{|}_-{U_B} & }
    \end{array}
    &
    \begin{array}{c}
      \xymatrix@-.5pc{
        \ar[r]|-@{|}^-{U_A} \ar@{=}[d] \ar@{}[dr]|\Downarrow &
        \ar[r]|-@{|}^-{\fhat} \ar[d]_f \ar@{}[dr]|\Downarrow
        & \ar@{=}[d]\\
        \ar[r]|-@{|}_-{\fhat} &
        \ar[r]|-@{|}_-{U_B} &}
    \end{array} &=
    \begin{array}{c}
      \xymatrix@-.5pc{
        \ar[r]|-@{|}^-{\fhat} \ar@{=}[d] \ar@{}[dr]|{\Downarrow 1_{\fhat}}
        & \ar@{=}[d]\\
        \ar[r]|-@{|}_-{\fhat} & }
    \end{array}
  \end{align}
  A \textbf{conjoint} of $f$, denoted $\fchk\maps B\hto A$, is a
  companion of $f$ in the double category $\lD^{h\cdot\mathrm{op}}$
  obtained by reversing the horizontal 1-cells, but not the vertical
  1-morphisms, of \lD.
\end{defn}
\noindent
In a pseudo double category, the second equation above requires an insertion of unit isomorphisms to make sense due to horizontal composition only holding up to isomorphism.
\begin{defn}
  We say that a double category is \textbf{fibrant} if every vertical
  1-morphism has both a companion and a conjoint.
\end{defn}

\begin{defn}
Let $\lA$ and $\lB$ be pseudo double categories. A \textbf{lax double functor} is a functor $F \colon\lA \to \lB$ that takes items of $\lA$ to items of $\lB$ of the corresponding type, respecting vertical composition in the strict sense and the horizontal composition up to an assigned comparison $\phi$. This means that we have functors $F_0 \colon \lA_0 \to \lB_0$ and $F_1 \colon \lA_1 \to \lB_1$ such that the following equations are satisfied: $$S \circ F_1 = F_0 \circ S$$ $$T \circ F_1 = F_0 \circ T$$ For brevity, we will omit the subscripts and simply say $F$. As to whether we mean $F_0$ or $F_1$ will be clear from context.

Also, every object $A$ is equipped with a special globular 2-morphism $\phi_{A} \colon 1_{F(A)} \to F(1_{A})$ (the identity comparison), and every horizontal composition $N_{1} \odot N_{2}$ is equipped with a special globular 2-morphism $\phi(N_{1},N_{2}) \colon F(N_{1}) \odot F(N_{2}) \to F(N_{1} \odot N_{2})$ (the composition comparison), in a coherent way. This means that the following diagrams commute.

\begin{enumerate}

\item For a horizontal composite, $\beta \star \alpha$,

\begin{equation}\label{eq:square}
  \xymatrix@-.5pc{
    F(A) \ar[r]|{|}^{F(N_{2})}  \ar[d] \ar@{}[dr]|{F(\alpha)}&
    F(B) \ar[d] \ar[r]|{|}^{F(N_{1})} \ar@{}[dr]|{F(\beta)}&
    F(C) \ar[d] &
     &
    F(A) \ar[r]|{|}^{F(N_{2})} \ar@{}[drr]|{\phi(N_{1},N_{2})} \ar[d]_{1} &
    F(B) \ar[r]|{|}^{F(N_{1})} &
    F(C) \ar[d]^{1} \\
    F(A') \ar[r]|{|}_{F(N_{4})} \ar@{}[drr]|{\phi(N_{3},N_{4})} \ar[d]_{1}&
    F(B') \ar[r]|{|}_{F(N_{3})} &
    F(C') \ar[d]^{1}&
    = &
    F(A) \ar[rr]|{|}^{F(N_{1} \odot N_{2})} \ar[d] \ar@{}[drr]|{F(\beta \star \alpha)}&
     &
    F(C) \ar[d] \\
    F(A') \ar[rr]|{|}_{F(N_{3} \odot N_{4})} & 
     & 
    F(C') &
     &
    F(A') \ar[rr]|{|}_{F(N_{3} \odot N_{4})} &
     &
    F(C') \\
  }.
\end{equation}

\item For a horizontal 1-cell $N \colon A \to B$, the following diagrams are commutative (under horizontal composition).

\[
\begin{tikzpicture}[scale=1.5]
\node (A) at (1,1) {$F(N) \odot 1_{F(A)}$};
\node (C) at (3,1) {$F(N)$};
\node (A') at (1,-1) {$F(N) \odot F(1_{A})$};
\node (C') at (3,-1) {$F(N \odot 1_{A})$};
\node (B) at (5,1) {$1_{F(B)} \odot F(N)$};
\node (B') at (5,-1) {$F(1_{B}) \odot F(N)$};
\node (D) at (7,1) {$F(N)$};
\node (D') at (7,-1) {$F(1_{B} \odot N)$};
\path[->,font=\scriptsize,>=angle 90]
(A) edge node[left]{$1 \odot \phi_{A}$} (A')
(C') edge node[right]{$F \rho$} (C)
(A) edge node[above]{$\rho_{F(N)}$} (C)
(A') edge node[above]{$\phi(N,1_{A})$} (C')
(B) edge node[left]{$\phi_{B} \odot 1$} (B')
(B') edge node[above]{$\phi(1_{B},N)$} (D')
(B) edge node[above]{$\lambda_{F(N)}$} (D)
(D') edge node[right]{$F \lambda$} (D);
\end{tikzpicture}
\]

\item For consecutive horizontal 1-cells $N_{1},N_{2}$ and $N_{3}$, the following diagram is commutative.

 \[\xymatrix{
    (F(N_{1}) \odot F(N_{2})) \odot F(N_{3}) \ar[r]^{a'}\ar[d]_{\phi(N_{1},N_{2}) \odot 1}
    & F(N_{1}) \odot (F(N_{2}) \odot F(N_{3})) \ar[d]^{1 \odot \phi(N_{2},N_{3})}\\
    F(N_{1} \odot N_{2}) \odot F(N_{3}) \ar[d]_{\phi(N_{1} \odot N_{2},N_{3})} &
    F(N_{1}) \odot F(N_{2} \odot N_{3}) \ar[d]^{\phi(N_{1},N_{2} \odot N_{3})}\\
    F((N_{1} \odot N_{2}) \odot N_{3})\ar[r]^{Fa} &
    F(N_{1} \odot (N_{2} \odot N_{3}))}\]
\end{enumerate}
\end{defn}

\begin{defn}
Let $\lA$ and $\lB$ be pseudo double categories and $\hat{F} \colon \lA \to \lB$ a lax double functor. An \textbf{oplax double functor} is a functor $F \colon \lA \to \lB^{co}$ such that $F$ and $\hat{F}$ agree on all objects, vertical 1-morphisms and horizontal 1-cells and where $\lB^{co}$ denotes the pseudo double category $\lB$ with all 2-morphisms reversed. In other words, we reverse the direction of the assigned comparison $\phi$ for horizontal composition in the definition of lax double functor.
\end{defn}


\begin{defn}
A lax double functor $F \colon \lC \to \lD$ between monoidal pseudo double categories is $\textbf{(lax) monoidal}$ if it is equipped with:
\begin{enumerate}
\item{a morphism $\epsilon \colon 1_{\lD} \to F(1_{\lC})$}
\item{a natural transformation $\mu_{A,B} \colon F(A) \otimes F(B) \to F(A \otimes B)$ for all objects $A$ and $B$ of $\lC$}
\item{a morphism $\delta \colon U_{1_{\lD}} \to F(U_{1_{\lC}})$}
\item{a natural transformation $\nu_{M,N} \colon F(M) \otimes F(N) \to F(M \otimes N)$ for all horizontal 1-cells $N$ and $M$ of $\lC$}
\end{enumerate}
such that the following diagrams commute: for objects $A,B$ and $C$ of $\lC$,
 \[\xymatrix{
    (F(A) \otimes F(B)) \otimes F(C) \ar[r]^{\alpha^\prime}\ar[d]_{\mu_{A,B} \otimes 1}
    & F(A) \otimes (F(B) \otimes F(C)) \ar[d]^{1 \otimes \mu_{B,C}}\\
    F(A \otimes B) \otimes F(C) \ar[d]_{\mu_{A \otimes B,C}} &
    F(A) \otimes F(B \otimes C) \ar[d]^{\mu_{A,B \otimes C}}\\
    F((A \otimes B) \otimes C)\ar[r]^{F\alpha} &
    F(A \otimes (B \otimes C))}\]
\[
\begin{tikzpicture}[scale=1.5]
\node (A) at (1,1) {$F(A) \otimes 1_{\lD}$};
\node (C) at (3,1) {$F(A)$};
\node (A') at (1,-1) {$F(A) \otimes F(1_{\lC})$};
\node (C') at (3,-1) {$F(A \otimes 1_{\lC})$};
\node (B) at (5,1) {$1_{\lD} \otimes F(A)$};
\node (B') at (5,-1) {$F(1_{\lC}) \otimes F(A)$};
\node (D) at (7,1) {$F(A)$};
\node (D') at (7,-1) {$F(1_{\lC} \otimes A)$};
\path[->,font=\scriptsize,>=angle 90]
(A) edge node[left]{$1 \otimes \epsilon$} (A')
(C') edge node[right]{$F(r_{A})$} (C)
(A) edge node[above]{$r_{F(A)}$} (C)
(A') edge node[above]{$\mu_{A,1_{\lC}}$} (C')
(B) edge node[left]{$\epsilon \otimes 1$} (B')
(B') edge node[above]{$\mu_{1_{\lC},A}$} (D')
(B) edge node[above]{$\ell_{F(A)}$} (D)
(D') edge node[right]{$F(\ell_{A})$} (D);
\end{tikzpicture}
\]
and for horizontal 1-cells $N_{1},N_{2}$ and $N_{3}$ of $\lC$,
 \[\xymatrix{
    (F(N_{1}) \otimes F(N_{2})) \otimes F(N_{3}) \ar[r]^{\alpha^\prime}\ar[d]_{\nu_{N_{1},N_{2}} \otimes 1}
    & F(N_{1}) \otimes (F(N_{2}) \otimes F(N_{3})) \ar[d]^{1 \otimes \nu_{N_{2},N_{3}}}\\
    F(N_{1} \otimes N_{2}) \otimes F(N_{3}) \ar[d]_{\nu_{N_{1} \otimes N_{2},N_{3}}} &
    F(N_{1}) \otimes F(N_{2} \otimes N_{3}) \ar[d]^{\nu_{N_{1},N_{2} \otimes N_{3}}}\\
    F((N_{1} \otimes N_{2}) \otimes N_{3})\ar[r]^{F\alpha} &
    F(N_{1} \otimes (N_{2} \otimes N_{3}))}\]
\[
\begin{tikzpicture}[scale=1.5]
\node (A) at (1,1) {$F(N_{1}) \otimes U_{1_{\lD}}$};
\node (C) at (3,1) {$F(N_{1})$};
\node (A') at (1,-1) {$F(N_{1}) \otimes F(U_{1_{\lC}})$};
\node (C') at (3,-1) {$F(N_{1} \otimes U_{1_{\lC}})$};
\node (B) at (5,1) {$U_{1_{\lD}} \otimes F(N_{1})$};
\node (B') at (5,-1) {$F(U_{1_{\lC}}) \otimes F(N_{1})$};
\node (D) at (7,1) {$F(N_{1})$};
\node (D') at (7,-1) {$F(U_{1_{\lC}} \otimes N_{1})$};
\path[->,font=\scriptsize,>=angle 90]
(A) edge node[left]{$1 \otimes \delta$} (A')
(C') edge node[right]{$F(r_{N_{1}})$} (C)
(A) edge node[above]{$r_{F(N_{1})}$} (C)
(A') edge node[above]{$\nu_{N_{1},U_{1_{\lC}}}$} (C')
(B) edge node[left]{$\delta \otimes 1$} (B')
(B') edge node[above]{$\nu_{U_{1_{\lC}},N_{1}}$} (D')
(B) edge node[above]{$\ell_{F(N_{1})}$} (D)
(D') edge node[right]{$F(\ell_{N_{1}})$} (D);
\end{tikzpicture}
\]
We also require that:
\begin{enumerate}
\item{We have equalities $S \circ F_1 = F_0 \circ S$ and $T \circ F_1 = F_0 \circ T$ of lax monoidal functors.}
\item{The composition constraints for the lax double functor $F$ are monoidal natural transformations.}
\end{enumerate}
\end{defn}
Note that our monoidal lax double functors laxly preserve both the tensor product and composition, so that we in fact have \emph{lax} monoidal lax double functors. We will simply say `lax monoidal double functor' to avoid repetitiveness.
\begin{defn}
A \textbf{braided  lax monoidal double functor} $F \colon \lC \to \lD$ between braided monoidal pseudo double categories is a lax monoidal double functor that makes the following diagrams commute for all objects $A$ and $B$ of $\lC$ and all horizontal 1-cells $M$ and $N$ of $\lC$.
\[
\begin{tikzpicture}[scale=1.5]
\node (A) at (1,1) {$F(A) \otimes F(B)$};
\node (C) at (3,1) {$F(B) \otimes F(A)$};
\node (A') at (1,-1) {$F(A \otimes B)$};
\node (C') at (3,-1) {$F(B \otimes A)$};
\node (B) at (5,1) {$F(M) \otimes F(N)$};
\node (B') at (5,-1) {$F(M \otimes N)$};
\node (D) at (7,1) {$F(N) \otimes F(M)$};
\node (D') at (7,-1) {$F(N \otimes M)$};
\path[->,font=\scriptsize,>=angle 90]
(A) edge node[left]{$\mu_{A,B}$} (A')
(C) edge node[right]{$\mu_{B,A}$} (C')
(A) edge node[above]{$\beta_{F(A),F(B)}$} (C)
(A') edge node[above]{$F(\beta_{A,B})$} (C')
(B) edge node[left]{$\nu_{M,N}$} (B')
(B') edge node[above]{$F(\beta_{M,N})$} (D')
(B) edge node[above]{$\beta_{F(M),F(N)}$} (D)
(D) edge node[right]{$\nu_{N,M}$} (D');
\end{tikzpicture}
\]

\end{defn}
\begin{defn}
A \textbf{symmetric lax monoidal double functor} is a braided lax monoidal double functor between symmetric monoidal pseudo double categories.
\end{defn}

\section{Main Results}

First we will construct a symmetric lax monoidal double functor $F^\prime \colon \text{Cospan}(\textbf{C}) \to \textbf{BD}$ where $\textbf{BD}$ is the symmetric monoidal category $\textbf{D}$ viewed as a one object bicategory and where Cospan$(\bold{C})$ is a symmetric monoidal bicategory whose objects are that of $\bold{C}$, morphisms are cospans and 2-morphisms are maps of cospans, where a map of cospans is a map between the apices of two cospans such that the resulting two adjacent triangles commute. We will do this in the following sequence of propostions.

\begin{prop}\label{Proposition 4.1}
There exists a pseudo double category $\textnormal{Cospan}$($\bold{C}$).
\end{prop}

\begin{proof}
Objects are given by objects of $\textbf{C}$ and vertical 1-morphisms are morphisms of $\textbf{C}$. Horizontal 1-cells are cospans in $\textbf{C}$ and 2-morphisms are triples of maps $(a, \phi, c)$ between two cospans in $\textbf{C}$ such that two adjacent commuting squares result.

\[
\begin{tikzpicture}[scale=1.5]
\node (A) at (1,1) {$A$};
\node (B) at (2,1) {$B$};
\node (C) at (3,1) {$C$};
\node (A') at (1,0) {$A'$};
\node (B') at (2,0) {$B'$};
\node (C') at (3,0) {$C'$};
\path[->,font=\scriptsize,>=angle 90]
(A) edge node[above]{$f$} (B)
(A') edge node[above]{$f^\prime$} (B')
(A) edge node[left]{$a$} (A')
(C) edge node[above]{$g$} (B)
(C') edge node[above]{$g'$} (B')
(B) edge node[left]{$\phi$} (B')
(C) edge node[right]{$c$} (C');
\end{tikzpicture}
\]
The source, target and unit functors $S, T$ and $U$, respecively, are obvious. The associator comes from the universal property of a pushout and the left unit law comes from $B+_{B}A$ and $A$ both being colimits of the span $B \xleftarrow{\id} B \rightarrow A$ and the right unit law is similar. By definition, this gives us a pseudo double category \textnormal{Cospan}($\bold{C}$).
\end{proof}

\begin{prop}\label{Proposition 4.2}
The pseudo double category $\textnormal{Cospan}(\bold{C})$ is symmetric monoidal.
\end{prop}

\begin{proof}
This follows from the definition of symmetric monoidal pseudo double category with the trivial cospan $0 \to 0 \leftarrow 0$ as the unit for the horizontal edge category and that we have isomorphisms between the cospans

\begin{center}
\begin{tikzpicture}[->,>=stealth',node distance=2cm, auto]
 \node (A) {$A_{1}+A_{2}$};
 \node (A1) [above of=A,right of=A] {$C_{1}+C_{2}$};
 \node (A2) [below of=A1,right of=A1] {$B_{1}+B_{2}$}; 
 \node (S) [above of=A2,right of=A2] {$D_{1}+D_{2}$};
 \node (B) [below of=S,right of=S] {$E_{1}+E_{2}$};
 \node (AS) [above of=A1,right of=A1] {$(C_{1}+C_{2}) +_{B_{1}+B_{2}} (D_{1}+D_{2})$};
 \draw[->] (A) to node [swap]{$$} (A1);
 \draw[->] (A2) to node {$$} (A1);
 \draw[->] (A2) to node [swap]{$$}(S);
\draw[->] (B) to node {}(S);
 \draw[->] (A1) to node [swap]{$$}(AS);
\draw[->] (S) to node {$$}(AS);
\end{tikzpicture}
\end{center}
and

\begin{center}
\begin{tikzpicture}[->,>=stealth',node distance=1.2cm, auto]
 \node(A1) {$A_{1}$};
 \node(C1) [above of=A1,right of=A1] {$C_{1}$};
 \node(B1) [below of=C1,right of=C1] {$B_{1}$};
 \node(D1) [above of=B1,right of=B1] {$D_{1}$};
 \node(E1) [below of=D1,right of=D1] {$E_{1}$};
 \node(A2) [right of=E1] {$A_{2}$};
 \node(C2) [above of=A2,right of=A2] {$C_{2}$};
 \node(B2) [below of=C2,right of=C2] {$B_{2}$};
 \node(D2) [above of=B2,right of=B2] {$D_{2}$};
 \node(E2) [below of=D2,right of=D2] {$E_{2}$};
 \node(C1+D1) [above of=C1,right of=C1] {$C_{1}+D_{1}$};
 \node(C2+D2) [above of=C2,right of=C2] {$C_{2}+D_{2}$};
 \node(C1pD1) [above of=C1+D1] {$C_{1} +_{B_{1}} D_{1}$};
 \node(C2pD2) [above of=C2+D2] {$C_{2} +_{B_{2}} D_{2}$};
 \node(C1pD1+C2pD2) [above right=1.4cm and 0cm of C1pD1] {$(C_{1} +_{B_{1}} D_{1}) + (C_{2} +_{B_{2}} D_{2})$};
 \draw[->] (A1) to node [swap]{$$} (C1);
 \draw[->] (B1) to node [swap]{$$} (C1);
 \draw[->] (B1) to node [swap]{$$} (D1);
 \draw[->] (E1) to node [swap]{$$} (D1);
 \draw[->] (A2) to node [swap]{$$} (C2);
 \draw[->] (B2) to node [swap]{$$} (C2);
 \draw[->] (B2) to node [swap]{$$} (D2);
 \draw[->] (E2) to node [swap]{$$} (D2);
 \draw[->] (C1) to node [swap]{$$} (C1+D1);
 \draw[->] (D1) to node [swap]{$$} (C1+D1);
 \draw[->] (C2) to node [swap]{$$} (C2+D2);
 \draw[->] (D2) to node [swap]{$$} (C2+D2);
 \draw[->] (C1+D1) to node [swap]{$$} (C1pD1);
 \draw[->] (C2+D2) to node [swap]{$$} (C2pD2);
 \draw[->] (C1pD1) to node [swap] {$$} (C1pD1+C2pD2);
 \draw[->] (C2pD2) to node [swap] {$$} (C1pD1+C2pD2);
\end{tikzpicture}.
\end{center}
Both $(C_{1} + C_{2}) +_{B_{1}+B_{2}} (D_{1} + D_{2})$ and $(C_{1} +_{B_{1}} D_{1}) + (C_{2} +_{B_{2}} D_{2})$ are colimits of the following diagram consisting of cospans $M_{1},M_{2}$ and $N_{1}, N_{2}$ with shared feet $B_{1}$ and $B_{2}$,  respectively

\begin{center}
\begin{tikzpicture}[->,>=stealth',node distance=1.1cm, auto]
 \node(A1) {$A_{1}$};
 \node(C1) [above of=A1,right of=A1] {$C_{1}$};
 \node(B1) [below of=C1,right of=C1] {$B_{1}$};
 \node(D1) [above of=B1,right of=B1] {$D_{1}$};
 \node(E1) [below of=D1,right of=D1] {$E_{1}$};
 \node(A2) [right of=E1] {$A_{2}$};
 \node(C2) [above of=A2,right of=A2] {$C_{2}$};
 \node(B2) [below of=C2,right of=C2] {$B_{2}$};
 \node(D2) [above of=B2,right of=B2] {$D_{2}$};
 \node(E2) [below of=D2,right of=D2] {$E_{2}$};
 \draw[->] (A1) to node [swap]{$$} (C1);
 \draw[->] (B1) to node [swap]{$$} (C1);
 \draw[->] (B1) to node [swap]{$$} (D1);
 \draw[->] (E1) to node [swap]{$$} (D1);
 \draw[->] (A2) to node [swap]{$$} (C2);
 \draw[->] (B2) to node [swap]{$$} (C2);
 \draw[->] (B2) to node [swap]{$$} (D2);
 \draw[->] (E2) to node [swap]{$$} (D2);
\end{tikzpicture}
\end{center}
\noindent
and this gives the globular isomorphism $\fx\maps (M_1\ten N_1)\odot (M_2\ten N_2)\too[\sim] (M_1\odot M_2)\ten (N_1\odot N_2)$. The globular isomorphism $\fx$ makes the rest of the diagrams in the definition straightforward. The pentagon and triangle equations are also straightforward; all of the maps are given by maps of cospans, which in the double category case, are triples of isomorphisms between cospans as in the diagram in the previous proposition.
\end{proof}

\begin{prop}\label{Proposition 4.3}
The one-object bicategory $\bold{BD}$ forms a pseudo double category with one object $\{*\}$, whose only vertical 1-morphism is the identity, whose horizontal 1-cells are objects of $\bold{D}$ and whose 2-morphisms are corresponding squares. We denote this double category as $\lB \lD$.
\end{prop}

\begin{proof}
We have one object which we denote as $\{*\}$, and trivial composition of vertical 1-morphisms as we only have $\id_{*} \colon \{*\} \to \{*\}$. The objects of $\bold{D}$ appear as horizontal 1-cells and 2-morphisms are boxes as in the chart below.

\begin{center}
\begin{tabular}{ |l|l| }
  \hline
  \multicolumn{2}{|c|}{$\lB \lD$} \\
  \hline
  Objects & $\{ * \}$ \\
  \multirow{-4}{*}{Vertical 1-morphisms} & \begin{tikzpicture}[->,>=stealth',node distance=1.1cm, auto]
 \node(a) {$*$};
 \node(b) [below of=a] {$*$};
 \draw[->] (a) to node {Identity morphism} (b);
\end{tikzpicture}\\
  Horizontal 1-cells & \begin{tikzpicture}[->,>=stealth',node distance=1.1cm, auto]
 \node(a) {$*$};
 \node(b) [right of=a] {$*$};
 \draw[->] (a) to node {$d$} (b);
\end{tikzpicture} Objects of $\textbf{D}$ \\
  \multirow{-5}{*}{2-morphisms} & \begin{tikzpicture}[->,>=stealth',node distance=1.1cm, auto]
 \node(a) {$*$};
 \node(b) [right of=a] {$*$};
 \node(c) [below of=a] {$*$};
 \node(d) [right of=c] {$*$};
 \draw[->] (a) to node {$d$} (b);
 \draw[->] (b) to node {$\id_{{*}}$} (d);
 \draw[->] (a) to node [swap]{$\id_{{*}}$} (c);
 \draw[->] (c) to node [swap]{$d'$} (d);
\end{tikzpicture} \\
  \hline
\end{tabular}
\end{center}

Composition of vertical 1-morphisms is trivial. For composition of horizontal 1-cells and 2-morphisms, we have

\begin{center}
\begin{tikzpicture}[->,>=stealth',node distance=1.3cm, auto]
 \node(a) {$*$};
 \node(b) [right of=a] {$*\quad \odot$};
 \draw[->] (a) to node {$d$} (b);
\end{tikzpicture} \begin{tikzpicture}[->,>=stealth',node distance=1.1cm, auto]
 \node(a) {$*$};
 \node(b) [right of=a] {$*$};
 \draw[->] (a) to node {$d'$} (b);
\end{tikzpicture}
=
\begin{tikzpicture}[->,>=stealth',node distance=1.1cm, auto]
 \node(a) {$*$};
 \node(b) [right of=a] {$*$};
 \draw[->] (a) to node {$d\otimes d'$} (b);
\end{tikzpicture}
\end{center}
and

\begin{center}
\begin{tikzpicture}[->,>=stealth',node distance=1.1cm, auto]
 \node(a) {$*$};
 \node(b) [right of=a] {$*$};
 \node(c) [below of=a] {$*$};
 \node(d) [right of=c] {$*$};
 \draw[->] (a) to node {$d_{1}$} (b);
 \draw[->] (b) to node {$\id_{{*}}\text{           }\text{    } \text{  } \odot$} (d);
 \draw[->] (a) to node [swap]{$\id_{{*}}$} (c);
 \draw[->] (c) to node [swap]{$d_{1}'$} (d);
\end{tikzpicture}
\begin{tikzpicture}[->,>=stealth',node distance=1.1cm, auto]
 \node(a) {$*$};
 \node(b) [right of=a] {$*$};
 \node(c) [below of=a] {$*$};
 \node(d) [right of=c] {$*$};
 \draw[->] (a) to node {$d_{2}$} (b);
 \draw[->] (b) to node {$\id_{{*}}\text{       }=$} (d);
 \draw[->] (a) to node [swap]{$\id_{{*}}$} (c);
 \draw[->] (c) to node [swap]{$d_{2}'$} (d);
\end{tikzpicture}
\begin{tikzpicture}[->,>=stealth',node distance=1.1cm, auto]
 \node(a) {$*$};
 \node(b) [right of=a] {$*$};
 \node(c) [below of=a] {$*$};
 \node(d) [right of=c] {$*$};
 \draw[->] (a) to node {$d_{1} \otimes d_{2}$} (b);
 \draw[->] (b) to node {$\id_{{*}}$} (d);
 \draw[->] (a) to node [swap]{$\id_{{*}}$} (c);
 \draw[->] (c) to node [swap]{$d_{1}' \otimes d_{2}'$} (d);
\end{tikzpicture}.
\end{center}
\end{proof}

\begin{prop}\label{Proposition 4.4}
The double category $\lB \lD$ is symmetric monoidal.
\end{prop}

\begin{proof}

Tensoring of types of objects corresponds with compositions as in the previous proposition. Namely, tensoring of objects and vertical 1-morphisms is trivial. For tensoring of horizontal 1-cells and 2-morphisms, we have

\begin{center}
\begin{tikzpicture}[->,>=stealth',node distance=1.3cm, auto]
 \node(a) {$*$};
 \node(b) [right of=a] {$*\quad \otimes$};
 \draw[->] (a) to node {$d$} (b);
\end{tikzpicture} \begin{tikzpicture}[->,>=stealth',node distance=1.1cm, auto]
 \node(a) {$*$};
 \node(b) [right of=a] {$*$};
 \draw[->] (a) to node {$d'$} (b);
\end{tikzpicture}
=
\begin{tikzpicture}[->,>=stealth',node distance=1.1cm, auto]
 \node(a) {$*$};
 \node(b) [right of=a] {$*$};
 \draw[->] (a) to node {$d\otimes d'$} (b);
\end{tikzpicture}
\end{center}
and

\begin{center}
\begin{tikzpicture}[->,>=stealth',node distance=1.1cm, auto]
 \node(a) {$*$};
 \node(b) [right of=a] {$*$};
 \node(c) [below of=a] {$*$};
 \node(d) [right of=c] {$*$};
 \draw[->] (a) to node {$d_{1}$} (b);
 \draw[->] (b) to node {$\id_{{*}}\text{           }\text{    } \text{  } \otimes$} (d);
 \draw[->] (a) to node [swap]{$\id_{{*}}$} (c);
 \draw[->] (c) to node [swap]{$d_{1}'$} (d);
\end{tikzpicture}
\begin{tikzpicture}[->,>=stealth',node distance=1.1cm, auto]
 \node(a) {$*$};
 \node(b) [right of=a] {$*$};
 \node(c) [below of=a] {$*$};
 \node(d) [right of=c] {$*$};
 \draw[->] (a) to node {$d_{2}$} (b);
 \draw[->] (b) to node {$\id_{{*}}\text{       }=$} (d);
 \draw[->] (a) to node [swap]{$\id_{{*}}$} (c);
 \draw[->] (c) to node [swap]{$d_{2}'$} (d);
\end{tikzpicture}
\begin{tikzpicture}[->,>=stealth',node distance=1.1cm, auto]
 \node(a) {$*$};
 \node(b) [right of=a] {$*$};
 \node(c) [below of=a] {$*$};
 \node(d) [right of=c] {$*$};
 \draw[->] (a) to node {$d_{1} \otimes d_{2}$} (b);
 \draw[->] (b) to node {$\id_{{*}}$} (d);
 \draw[->] (a) to node [swap]{$\id_{{*}}$} (c);
 \draw[->] (c) to node [swap]{$d_{1}' \otimes d_{2}'$} (d);
\end{tikzpicture}
\end{center}
The unit for horizontal composition is given by $* \xrightarrow{I} *$ where $I$ is the unit object of $\bold{D}$ and the unit for 2-morphisms is

\begin{center}
\begin{tikzpicture}[->,>=stealth',node distance=1.1cm, auto]
 \node(a) {$*$};
 \node(b) [right of=a] {$*$};
 \node(c) [below of=a] {$*$};
 \node(d) [right of=c] {$*$};
 \draw[->] (a) to node {$\id_{*}$} (b);
 \draw[->] (b) to node {$\id_{*}$} (d);
 \draw[->] (a) to node [swap]{$\id_{*}$} (c);
 \draw[->] (c) to node [swap]{$\id_{*}$} (d);
\end{tikzpicture}
\end{center}
As $\bold{D}$ is symmetric monoidal, it follows that $\lB \lD$ is symmetric monoidal.
\end{proof}

\begin{prop}\label{Proposition 4.5}
There exists a symmetric lax monoidal double functor $F^\prime \colon \textnormal{Cospan}(\bold{C}) \to \lB \lD$.
\end{prop}

\begin{proof}
The functor $F^\prime$ maps every object $c \in$ Ob(Cospan($\bold{C}$)) to the one object $\{*\}$ of $\lB \lD$ and vertical 1-morphisms of $\bold{C}$ map to the vertical 1-morphism $I$ in $\lB \lD$, which is the identity morphism on the single object $\{*\}$ of $\lB \lD$. Horizontal 1-cells of Cospan$(\bold{C})$, which are cospans in $\bold{C}$, map to the horizontal 1-cell $F(c)$, where $c$ is the apex of a cospan in $\textbf{C}$ and $F(c)$ is an object of $\bold{D}$. Then a 2-morphism, which is a triple of maps between between two cospans in $\textbf{C}$, maps to the 2-morphism which is a box with top $F(c)$, left and right sides $I$ and bottom $F(d)$ where $d$ is the apex of the second cospan in $\textbf{C}$. 

\[
\begin{tikzpicture}[scale=1.5]
\node (A) at (1,1) {$a$};
\node (B) at (2,1) {$c$};
\node (C) at (3,1) {$b$};
\node (A') at (1,0) {$a'$};
\node (B') at (2,0) {$d$};
\node (C') at (3,0) {$b'$};
\node (m) at (4,0.5) {$\mapsto$};
\node (D) at (5,0) {$*$};
\node (E) at (5,1) {$*$};
\node (F) at (6,0) {$*$};
\node (G) at (6,1) {$*$};
\node (H) at (5.5,0.5) {$\Swarrow F(h)$};
\path[->,font=\scriptsize,>=angle 90]
(A) edge node[above]{$$} (B)
(A') edge node[above]{$$} (B')
(A) edge node[left]{$f$} (A')
(C) edge node[above]{$$} (B)
(C') edge node[above]{$$} (B')
(B) edge node[left]{$h$} (B')
(C) edge node[right]{$g$} (C')
(E) edge node[left]{$\id_{*}$} (D)
(G) edge node[right]{$\id_{*}$} (F)
(E) edge node[above]{$F(c)$} (G)
(D) edge node[below]{$F(d)$} (F);
\end{tikzpicture}
\]
If $(f,h,g)$ is the underlying 2-morphism in Cospan$(\bold{C})$, this gives us $F(h) \colon I \otimes F(c) \to F(d) \otimes I$, or just $F(h) \colon F(c) \to F(d)$. For notational purposes, we will consider the following cospans and pushouts with shared feet:

\begin{center}
\begin{tikzpicture}[->,>=stealth',node distance=1.1cm, auto]
 \node(X) {$X$};
 \node(N1) [above of=X,right of=X] {$N_{1}$};
 \node(Y) [below of=N1,right of=N1] {$Y$};
 \node(N2) [above of=Y,right of=Y] {$N_{2}$};
 \node(Z) [below of=N2,right of=N2] {$Z$};
 \node(N3) [above of=Z,right of=Z] {$N_{3}$};
 \node(W) [below of=N3,right of=N3] {$W$};
 \node(N1+N2) [above of=N1,right of=N1] {$N_{1}+_{Y}N_{2}$};
 \node(N2+N3) [above of=N2,right of=N2] {$N_{2}+_{Z}N_{3}$};
 \draw[->] (A1) to node [swap]{$$} (C1);
 \draw[->] (B1) to node [swap]{$$} (C1);
 \draw[->] (B1) to node [swap]{$$} (D1);
 \draw[->] (E1) to node [swap]{$$} (D1);
 \draw[->] (N1) to node [swap]{$$} (N1+N2);
 \draw[->] (N2) to node [swap]{$$} (N1+N2);
 \draw[->] (N2) to node [swap]{$$} (N2+N3);
 \draw[->] (Z) to node [swap]{$$} (N3);
 \draw[->] (N3) to node [swap]{$$} (N2+N3);
 \draw[->] (W) to node [swap]{$$} (N3);
\end{tikzpicture}
\end{center}
As $F \colon (\textbf{C},+,0) \to (\textbf{D},\otimes,I)$ is symmetric lax monoidal, we have morphisms $e \colon I \to F(0)$ and $\phi_{N_{1},N_{2}} \colon F(N_{1}) \otimes F(N_{2}) \to F(N_{1} \otimes N_{2})$ such that the following diagrams commute:

 \[\xymatrix{
    (F(N_{1}) \otimes F(N_{2})) \otimes F(N_{3}) \ar[r]^{a'}\ar[d]_{\phi_{N_{1},N_{2}} \otimes 1}
    & F(N_{1}) \otimes (F(N_{2}) \otimes F(N_{3})) \ar[d]^{1 \otimes \phi_{N_{2},N_{3}}}\\
    F(N_{1} + N_{2}) \otimes F(N_{3}) \ar[d]_{\phi_{N_{1}+N_{2},N_{3}}} &
    F(N_{1}) \otimes F(N_{2} + N_{3}) \ar[d]^{\phi_{N_{1},N_{2}+N_{3}}}\\
    F((N_{1}+N_{2})+N_{3})\ar[r]^{Fa} &
    F(N_{1}+(N_{2}+N_{3}))}\]

 \[\vcenter{\xymatrix{
      F(N_{1}) \otimes I\ar[r]^{1 \otimes e}\ar[d]_{\rho} &
      F(N_{1}) \otimes F(0) \ar[d]^{\phi_{N_{1},0}}\\
      F(N_{1}) \ar@{<-}[r]^{F_{\rho}} &
      F(N_{1}+0) }}\]

 \[\vcenter{\xymatrix{
      I \otimes F(N_{1})\ar[r]^{e \otimes 1}\ar[d]_{\lambda} &
      F(0) \otimes F(N_{1}) \ar[d]^{\phi_{0,N_{1}}}\\
      F(N_{1}) \ar@{<-}[r]^{F_{\lambda}} &
      F(0+N_{1}) }}\]
Our goal is to prove that the corresponding hexagon for $F^\prime \colon \text{Cospan}(\textbf{C}) \to \lB \lD$ commutes:

 \[\xymatrix{
    (F(N_{1}) \otimes F(N_{2})) \otimes F(N_{3}) \ar[r]^{a'}\ar[d]_{(F(J_{N_{1},N_{2}}) \circ \phi_{N_{1},N_{2}}) \otimes 1}
    & F(N_{1}) \otimes (F(N_{2}) \otimes F(N_{3})) \ar[d]^{1 \otimes (F(J_{N_{2},N_{3}}) \circ \phi_{N_{2},N_{3}})}\\
    F(N_{1} +_{Y} N_{2}) \otimes F(N_{3}) \ar[d]_{F(J_{N_{1}+_{Y}N_{2},N_{3}}) \circ \phi_{N_{1}+_{Y}N_{2},N_{3}}} &
    F(N_{1}) \otimes F(N_{2} +_{Z} N_{3}) \ar[d]^{F(J_{N_{1},N_{2}+_{Z}N_{3}}) \circ \phi_{N_{1},N_{2}+_{Z}N_{3}}}\\
    F((N_{1}+_{Y}N_{2})+_{Z}N_{3})\ar[r]^{F(a)} &
    F(N_{1}+_{Y}(N_{2}+_{Z}N_{3}))}\]

We do this by realizing the hexagon as one of the horizontal faces, say the bottom, of a hexagonal prism, all of whose sides commute and whose top is the commutative hexagon that comes from $F \colon (\textbf{C},+,0) \to (\textbf{D},\otimes,I)$ being symmetric lax monoidal. Denoting the coequalizer maps $J_{N_{1}+_{Y}N_{2}} \colon N_{1}+N_{2} \to N_{1}+_{Y}N_{2}$ and $J_{N_{2}+_{Z}N_{3}} \colon N_{2}+N_{3} \to N_{2}+_{Z}N_{3}$ simply as $J$, consider the following hexagonal prism:

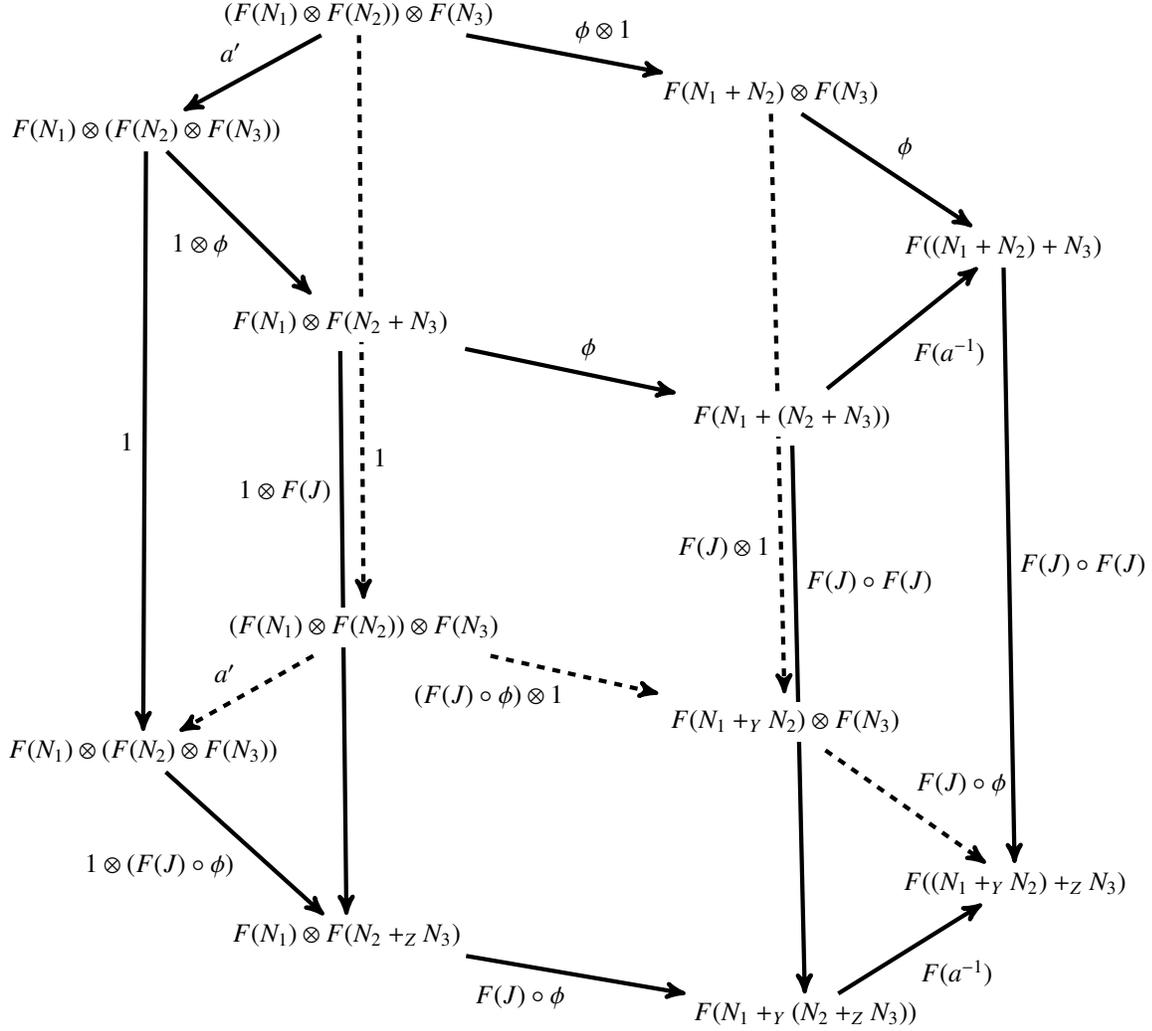
\begin{center}
\hspace*{-2cm}%
\begin{tikzpicture}[->,>=stealth',node distance=1.2cm, auto,scale=.80]
 \node(A1) {$(F(N_{1}) \otimes F(N_{2})) \otimes F(N_{3})$};
 \node(A2) [below right=0.5cm and 2cm of A1] {$F(N_{1}+N_{2}) \otimes F(N_{3})$};
 \node(A3) [below right=1.5cm and 0.1cm of A2] {$F((N_{1}+N_{2})+N_{3})$};
 \node(A4) [below left=1cm and -1cm of A1] {$F(N_{1}) \otimes (F(N_{2}) \otimes F(N_{3}))$};
 \node(A5) [below right=1.9cm and -1cm of A4] {\colorbox{white}{$F(N_{1}) \otimes F(N_{2}+N_{3})$}};
 \node(A6) [below right=0.5cm and 2.8cm of A5] {\colorbox{white}{$F(N_{1}+(N_{2}+N_{3}))$}};
 \node(B1) [below left=7.5cm and -4cm of A1] {\colorbox{white}{$(F(N_{1}) \otimes F(N_{2})) \otimes F(N_{3})$}};
 \node(B2) [below right=0.5cm and 1.83cm of B1] {\colorbox{white}{$F(N_{1}+_{Y}N_{2}) \otimes F(N_{3})$}};
 \node(B3) [below right=1.5cm and -0.3cm of B2] {$F((N_{1}+_{Y}N_{2})+_{Z}N_{3})$};
 \node(B4) [below left=1cm and -1cm of B1] {$F(N_{1}) \otimes (F(N_{2}) \otimes F(N_{3}))$};
 \node(B5) [below right=1.9cm and -0.85cm of B4] {$F(N_{1}) \otimes F(N_{2}+_{Z}N_{3})$};
 \node(B6) [below right=0.5cm and 2.85cm of B5] {$F(N_{1}+_{Y}(N_{2}+_{Z}N_{3}))$};
 \draw[ultra thick,->] (A1) to node {$\phi \otimes 1$} (A2);
 \draw[ultra thick,->] (A2) to node {$\phi$} (A3);
 \draw[ultra thick,->] (A1) to node [swap]{$a'$} (A4);
 \draw[ultra thick,->] (A4) to node [swap]{$1 \otimes \phi$} (A5);
 \draw[ultra thick,->] (A5) to node {$\phi$} (A6);
 \draw[ultra thick,->] (A6) to node [swap]{$F(a^{-1})$} (A3);
 \draw[dashed,ultra thick,->] (B1) to node [swap]{$(F(J)\circ \phi) \otimes 1$} (B2);
 \draw[dashed,ultra thick,->] (B2) to node {$F(J) \circ \phi$} (B3);
 \draw[dashed,ultra thick,->] (B1) to node [swap]{$a'$} (B4);
 \draw[ultra thick,->] (B4) to node [swap]{$1 \otimes (F(J) \circ \phi)$} (B5);
 \draw[ultra thick,->] (B5) to node [swap]{$F(J) \circ \phi$} (B6);
 \draw[ultra thick,->] (B6) to node [swap]{$F(a^{-1})$} (B3);
 \draw[dashed,ultra thick,->] (A1) to node [near end]{$1$} (B1);
 \draw[dashed,ultra thick,->] (A2) to node [swap,near end]{$F(J) \otimes 1$} (B2);
 \draw[ultra thick,->] (A3) to node {$F(J) \circ F(J)$} (B3);
 \draw[ultra thick,->] (A4) to node [swap]{$1$} (B4);
 \draw[ultra thick,->] (A5) to node [swap,near start]{$1 \otimes F(J)$} (B5);
 \draw[ultra thick,->] (A6) to node [near start] {$F(J) \circ F(J)$} (B6);
 \node(A5) [below right=1.9cm and -1cm of A4] {\colorbox{white}{$F(N_{1}) \otimes F(N_{2}+N_{3})$}};
 \node(A6) [below right=0.5cm and 2.8cm of A5] {\colorbox{white}{$F(N_{1}+(N_{2}+N_{3}))$}};
 \node(B1) [below left=7.5cm and -4cm of A1] {\colorbox{white}{$(F(N_{1}) \otimes F(N_{2})) \otimes F(N_{3})$}};
 \node(B2) [below right=0.5cm and 1.83cm of B1] {\colorbox{white}{$F(N_{1}+_{Y}N_{2}) \otimes F(N_{3})$}};
 \end{tikzpicture}
 \end{center}
The six lateral sides are given as follows:

\begin{center}
\begin{tikzpicture}[->,>=stealth',node distance=1.1cm, auto]
 \node(X) {$(F(N_{1}) \otimes F(N_{2})) \otimes F(N_{3})$};
 \node(Y) [right=5cm of X] {$F(N_{1}) \otimes (F(N_{2}) \otimes F(N_{3}))$};
 \node(Z) [below=2cm of X] {$(F(N_{1}) \otimes F(N_{2})) \otimes F(N_{3})$};
 \node(W) [right=5cm of Z] {$F(N_{1}) \otimes (F(N_{2}) \otimes F(N_{3}))$};
 \draw[->] (X) to node {$a'$} (Y);
 \draw[->] (X) to node [swap]{$1$} (Z);
 \draw[->] (Y) to node {$1$} (W);
 \draw[->] (Z) to node {$a'$} (W);
\end{tikzpicture}
\end{center}

\begin{center}
\begin{tikzpicture}[->,>=stealth',node distance=1.1cm, auto]
 \node(X) {$F(N_{1}) \otimes (F(N_{2}) \otimes F(N_{3}))$};
 \node(Y) [right=5.1cm of X] {$F(N_{1}) \otimes F(N_{2}+N_{3})$};
 \node(Z) [below=2cm of X] {$F(N_{1}) \otimes (F(N_{2}) \otimes F(N_{3}))$};
 \node(W) [right=5cm of Z] {$F(N_{1}) \otimes F(N_{2}+_{Z}N_{3})$};
 \draw[->] (X) to node {$1 \otimes \phi_{N_{2},N_{3}}$} (Y);
 \draw[->] (X) to node [swap]{$1$} (Z);
 \draw[->] (Y) to node {$1 \otimes F(J_{N_{2},N_{3}})$} (W);
 \draw[->] (Z) to node {$1 \otimes (F(J_{N_{2},N_{3}}) \circ \phi_{N_{2},N_{3}})$} (W);
\end{tikzpicture}
\end{center}

\begin{center}
\begin{tikzpicture}[->,>=stealth',node distance=1.1cm, auto]
 \node(X) {$(F(N_{1}) \otimes F(N_{2})) \otimes F(N_{3})$};
 \node(Y) [right=5.1cm of X] {$F(N_{1}+N_{2}) \otimes F(N_{3})$};
 \node(Z) [below=2cm of X] {$(F(N_{1}) \otimes F(N_{2})) \otimes F(N_{3})$};
 \node(W) [right=5cm of Z] {$F(N_{1}+_{Y}N_{2}) \otimes F(N_{3})$};
 \draw[->] (X) to node {$\phi_{N_{1},N_{2}} \otimes 1$} (Y);
 \draw[->] (X) to node [swap]{$1$} (Z);
 \draw[->] (Y) to node {$F(J_{N_{1},N_{2}}) \otimes 1$} (W);
 \draw[->] (Z) to node {$(F(J_{N_{1},N_{2}}) \circ \phi_{N_{1},N_{2}}) \otimes 1$} (W);
\end{tikzpicture}
\end{center}

\begin{center}
\begin{tikzpicture}[->,>=stealth',node distance=1.1cm, auto]
 \node(X) {$F(N_{1}) \otimes F(N_{2}+N_{3})$};
 \node(Y) [right=5.25cm of X] {$F(N_{1}+(N_{2}+N_{3}))$};
 \node(Z) [below=2cm of X] {$F(N_{1}) \otimes F(N_{2}+_{Z}N_{3})$};
 \node(W) [right=5cm of Z] {$F(N_{1}+_{Y}(N_{2}+_{Z}N_{3}))$};
 \draw[->] (X) to node {$\phi_{N_{1},N_{2}+N_{3}}$} (Y);
 \draw[->] (X) to node [swap]{$1 \otimes F(J_{N_{2},N_{3}})$} (Z);
 \draw[->] (Y) to node [swap]{$F(J_{N_{1},N_{2}+_{Z}N_{3}}) \circ (1 + F(J_{N_{2},N_{3}}))$} (W);
 \draw[->] (Z) to node {$F(J_{N_{1},N_{2}+_{Z}N_{3}}) \circ \phi_{N_{1},N_{2}+_{Z}N_{3}}$} (W);
\end{tikzpicture}
\end{center}

\begin{center}
\begin{tikzpicture}[->,>=stealth',node distance=1.1cm, auto]
 \node(X) {$F(N_{1}+N_{2}) \otimes F(N_{3})$};
 \node(Y) [right=5.25cm of X] {$F((N_{1}+N_{2})+N_{3})$};
 \node(Z) [below=2cm of X] {$F(N_{1}+_{Y}N_{2}) \otimes F(N_{3})$};
 \node(W) [right=5cm of Z] {$F((N_{1}+_{Y}N_{2})+_{Z}N_{3})$};
 \draw[->] (X) to node {$\phi_{N_{1},N_{2}+N_{3}}$} (Y);
 \draw[->] (X) to node [swap]{$F(J_{N_{1},N_{2}}) \otimes 1$} (Z);
 \draw[->] (Y) to node [swap]{$F(J_{N_{1}+_{Y}N_{2},N_{3}}) \circ (F(J_{N_{1},N_{2}}) + 1)$} (W);
 \draw[->] (Z) to node {$F(J_{N_{1}+_{Y}N_{2},N_{3}}) \circ \phi_{N_{1}+_{Y}N_{2},N_{3}}$} (W);
\end{tikzpicture}
\end{center}

\begin{center}
\begin{tikzpicture}[->,>=stealth',node distance=1.1cm, auto]
 \node(X) {$F(N_{1}+(N_{2}+N_{3}))$};
 \node(Y) [right=5.35cm of X] {$F((N_{1}+N_{2})+N_{3})$};
 \node(Z) [below=2cm of X] {$F(N_{1}+_{Y}(N_{2}+_{Z}N_{3}))$};
 \node(W) [right=5cm of Z] {$F((N_{1}+_{Y}N_{2})+_{Z}N_{3})$};
 \draw[->] (X) to node {$F(a)$} (Y);
 \draw[->] (X) to node {$F(J_{N_{1},N_{2}+_{Z}N_{3}}) \circ (1+F(J_{N_{2},N_{3}})$} (Z);
 \draw[->] (Y) to node {$F(J_{N_{1}+_{Y}N_{2},N_{3}}) \circ (F(J_{N_{1},N_{2}})+1)$} (W);
 \draw[->] (Z) to node {$F(a)$} (W);
\end{tikzpicture}
\end{center}

The first diagram commutes trivially and the second and third diagrams commute by inspection. The fourth and fifth diagrams commute by naturality of $F$ and the last commutes by universality of coequalizers. As the top face and six lateral sides of the hexagonal prism commute, the bottom hexagon commutes as well. The two diagrams involving the left and right unitors commute because $F \colon \bold{C} \to \bold{D}$ is symmetric lax monoidal. It follows that $F^\prime \colon \text{Cospan}(\bold{C}) \to \lB \lD$ is a lax double functor.

Define $\epsilon \colon 1_{\lB \lD} \to F(1_{\text{Cospan}(\bold{C})})$ and $\mu_{A,B} \colon F(A) \otimes F(B) \to F(A \otimes B)$, where $A$ and $B$ are objects of Cospan$(\bold{C})$, both to be the identity, as $\lB \lD$ has only one vertical 1-morphism, the identity of its only object. As $F^\prime$ acts as $F$ on horizontal 1-cells, define $\delta \colon U_{1_{\lB \lD}} \to F(U_{1_{\text{Cospan}(\bold{C})}})$ and $\nu_{M,N} \colon F(M) \otimes F(N) \to F(M \otimes N)$, where $M$ and $N$ are horizontal 1-cells of Cospan$(\bold{C})$, to be the maps $e$ and $\phi$, respectively, that arise from the functor $F \colon \bold{C} \to \bold{D}$ being symmetric lax monoidal. Then all of the required diagrams for the lax double functor $F^\prime \colon \text{Cospan}(\bold{C}) \to \lB \lD$ to be symmetric monoidal commute, as the four diagrams involving objects and vertical 1-morphisms are trivial and the remaining four diagrams involving horizontal 1-cells are precisely the diagrams that commute because $F \colon \bold{C} \to \bold{D}$ is symmetric lax monoidal.
\end{proof}

\begin{defn}
Let $f \colon \textbf{C} \to \textbf{E}$ and $g \colon \textbf{D} \to \textbf{E}$ be functors with a common codomain. Then their $\textbf{comma category}$ is the category $(f/g)$ whose
\begin{enumerate}
\item{Objects are triples $(c,d,\alpha)$ where $c \in \textbf{C}, d \in \textbf{D}$ and $\alpha \colon f(c) \to g(d)$ is a morphism in $\textbf{E}$, and whose}
\item{Morphisms from $(c_{1},d_{1},\alpha_{1})$ to $(c_{2},d_{2},\alpha_{2})$ are pairs $(\beta, \gamma)$ where $\beta \colon c_{1} \to c_{2}$ and $\gamma \colon d_{1} \to d_{2}$ are morphisms in $\textbf{C}$ and $\textbf{D}$, respectively, such that the following diagram commutes}
\end{enumerate}

\[
\begin{tikzpicture}[scale=1.5]
\node (A) at (1,1) {$f(c_{1})$};
\node (B) at (2,1) {$f(c_{2})$};
\node (A') at (1,0) {$g(d_{1})$};
\node (B') at (2,0) {$g(d_{2})$};
\path[->,font=\scriptsize,>=angle 90]
(A) edge node[above]{$f(\beta)$} (B)
(A') edge node[above]{$g(\gamma)$} (B')
(A) edge node[left]{$\alpha_{1}$} (A')
(B) edge node[right]{$\alpha_{2}$} (B');
\end{tikzpicture}
\]
\end{defn}




\begin{prop}
Let $F_{1} \colon \lC \to \lE$ be an oplax double functor and $F_{2} \colon \lD \to \lE$ be a lax double functor where $\lC = (\bold{C}_{0},\bold{C}_{1}), \lD = (\bold{D}_{0}, \bold{D}_{1})$ and $\lE = (\bold{E}_{0}, \bold{E}_{1})$ are pseudo  double categories with $\bold{C}_{0}, \bold{C}_{1}$ the category of objects and category of arrows of the double category $\lC$, respectively, and similary for $\bold{D}_{0}, \bold{D}_{1}, \bold{E}_{0}$ and $\bold{E}_{1}$. Then there is a pseudo double category $(F_{1} / F_{2})$ consisting of a category of objects $\bold{A}_{0}$ and category of arrows $\bold{A}_{1}$ such that $\bold{A}_{0}$ is the comma category obtained from $F_{1} \colon \bold{C}_{0} \to \bold{E}_{0}$ and $F_{2} \colon \bold{D}_{0} \to \bold{E}_{0}$ and $\bold{A}_{1}$ is the comma category obtained from $F_{1} \colon \bold{C}_{1} \to \bold{E}_{1}$ and $F_{2} \colon \bold{D}_{1} \to \bold{E}_{1}$. We call $(F_{1}/F_{2})$ a \textnormal{\textbf{pseudo comma double category}}.
\end{prop}

\begin{proof}
The four different types of data, namely the objects and morphisms in both the category of objects and category of arrows, are obtained as prescribed by the definition of comma category. That these four types of data then fit together to form a pseudo double category then follows as such. Objects are given by triples $(c,d,\alpha)$ where $c \in \textbf{C}_{0}, d \in \textbf{D}_{0}$ and $\alpha \colon F_{1}(c) \to F_{2}(d)$ a morphism in $\textbf{E}_{0}$, and a vertical 1-morphism between two triples $(c_{1},d_{1},\alpha_{1})$ and $(c_{2},d_{2},\alpha_{2})$ are morphism pairs $(\beta,\gamma)$ where $\beta \colon c_{1} \to c_{2}$ and $\gamma \colon d_{1} \to d_{2}$ are morphisms in $\textbf{C}_{0}$ and $\textbf{D}_{0}$, respectively, such that the above square commutes in $\textbf{E}_{0}$. That composition of vertical 1-morphisms is strictly associative follows from composition of vertical 1-morphisms in $\textbf{C}_{0}$ and $\textbf{D}_{0}$ being strictly associative.

Similarly, we have that objects in the category of arrows, which are horizontal 1-cells, are also given as triples and composition of these triples is associatve only up to natural isomorphism. This follows from composition of horizontal 1-cells in $\lC$ and $\lD$ being associative only up to natural isomorphism. Abusing notation, if we have two horizontal 1-cells $(M,M^\prime, \alpha \colon F_{1}(M) \rightarrow F_{2}(M^\prime))$ and $(N,N^\prime, \alpha^\prime \colon F_{1}(N) \rightarrow F_{2}(N^\prime))$ where $\alpha$ and $\alpha^\prime$ are 2-morphisms, then we obtain a 2-morphism $F_{1}(N \odot M) \rightarrow F_{2}(N^\prime \odot M^\prime)$ by considering the following diagram.

\[
\begin{tikzpicture}[scale=1.5]
\node (A) at (1,2) {$F_{1}(c_{1})$};
\node (B) at (3,2) {$F_{1}(c_{2})$};
\node (C) at (5,2) {$F_{1}(c_{3})$};
\node (A') at (1,0) {$F_{2}(d_{1})$};
\node (B') at (3,0) {$F_{2}(d_{2})$};
\node (C') at (5,0) {$F_{2}(d_{3})$};
\node (E) at (3,2.75) {$\Downarrow \psi_{N,M}$};
\node (F) at (3,-0.75) {$\Downarrow \phi_{N^\prime, M^\prime}$};
\node (G) at (2,1) {$\Downarrow \alpha$};
\node (H) at (4,1) {$\Downarrow \alpha^\prime$};
\path[->,font=\scriptsize,>=angle 90]
(A) edge node[above]{$F_{1}(M)$} (B)
(A) edge[bend left=90] node [above]{$F_{1}(N \odot M)$}(C)
(A') edge[bend right=90] node [below]{$F_{2}(N^\prime \odot M^\prime)$}(C')
(A') edge node[above]{$F_{2}(M^\prime)$} (B')
(A) edge node[left]{$$} (A')
(B) edge node[above]{$F_{1}(N)$} (C)
(B') edge node[above]{$F_{2}(N^\prime)$} (C')
(B) edge node[left]{$$} (B')
(C) edge node[right]{$$} (C');
\end{tikzpicture}
\]
This gives us the desired 2-morphism. The remaining details are routine.
\end{proof}

It is worth noting the importance of the functors $F_{1}$ and $F_{2}$ in the above proposition being oplax and lax, respectively. This is precisely what allows the maps $\psi_{N,M} \colon F_{1}(N \odot M) \to F_{1}(N) \odot F_{1}(M)$ and $\phi_{N^\prime,M^\prime} \colon F_{2}(N^\prime) \odot F_{2}(M^\prime) \to F_{2}(N^\prime \odot M^\prime)$ to go in the proper directions.

\begin{thm}\label{Theorem 4.8}
The pseudo comma double category $(E/F^\prime)$ is the symmetric monoidal double category of $F$-decorated cospans in $\bold{C}$, where $E \colon \bold{1} \to \lB \lD$ is the symmetric oplax monoidal double functor that picks out the unit object of $\bold{D}$.
\end{thm}

\begin{proof}
We will verify the definition of the pseudo comma double category $(E/F^\prime)$, as $\bold{1}, \lB \lD$ and Cospan$(\bold{C})$ are symmetric monoidal double categories and we wish to show that the comma category $(E/F^\prime)$ is also a symmetric monoidal double category. Objects are given by triples $(*,c,\id_{ \{*\}})$ as $\bold{1}$ only has one object $\{*\}$ and $\lB \lD$ only has the identity on $\{*\}$ for vertical 1-morphisms, and so this triple is really just an object of $\bold{C}$ due to the triviality of the structure of $\{*\}$ and $\id_{ \{* \} }$. 
Vertical 1-morphisms of $(E/F^\prime)$ are given by pairs $(\id_{1},f)$ where $f$ is a morphism in $\bold{C}$, and hence vertical 1-morphsims are morphisms in $\bold{C}$. This gives us the objects and morphisms of the category of objects of $(E/F^\prime)$. 

For the category of arrows of $(E/F^\prime)$, objects are given by triples $(\id_{\bold{1}},a \to c \leftarrow b, f \colon I \to F(c))$ since $\bold{1}$ only has an identity for horizontal 1-cells and $F^\prime$ acts as $F$ on horizontal 1-cells. Thus objects in the category of arrows of $(E/F^\prime)$, which are horizontal 1-cells, are $F$-decorated cospans in $\bold{C}$. Morphisms in the category of arrows of $(E/F^\prime)$, which are the same as 2-morphisms of $(E/F^\prime)$, are pairs $(\id_{I},(f,h,g))$ such that following diagrams commute:

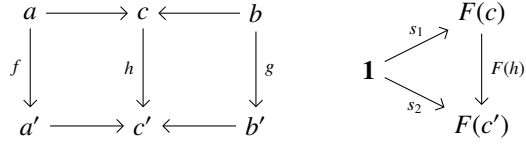
\[
\begin{tikzpicture}[scale=1.5]
\node (A) at (1,1) {$a$};
\node (B) at (2,1) {$c$};
\node (C) at (3,1) {$b$};
\node (A') at (1,0) {$a'$};
\node (B') at (2,0) {$c'$};
\node (C') at (3,0) {$b'$};
\node (D) at (4,0.5) {$\bold{1}$};
\node (E) at (5,0) {$F(c')$};
\node (F) at (5,1) {$F(c)$};
\path[->,font=\scriptsize,>=angle 90]
(A) edge node[above]{$$} (B)
(A') edge node[above]{$$} (B')
(A) edge node[left]{$f$} (A')
(C) edge node[above]{$$} (B)
(C') edge node[above]{$$} (B')
(B) edge node[left]{$h$} (B')
(C) edge node[right]{$g$} (C')
(D) edge node[below]{$s_{2}$} (E)
(D) edge node[above]{$s_{1}$} (F)
(F) edge node[right]{$F(h)$} (E);
\end{tikzpicture}
\]

These are precisely maps between apices of $F$-decorated cospans in $\bold{C}$. Thus we have that objects of $(E/F^\prime)$ are given by objects of $\bold{C}$, vertical 1-morphisms are given by morphisms of $\bold{C}$, horizontal 1-cells are given by $F$-decorated cospans in $\bold{C}$ and 2-morphisms are given by maps between $F$-decorated cospans in $\bold{C}$ such that the above two diagrams commute. That these four pieces of data fit together to form a pseudo double category follows readily from the definition. Moreover, we have that Cospan$(\bold{C})$ and $\lB \lD$ are symmetric monoidal pseudo double categories by Proposition \ref{Proposition 4.2} and Proposition \ref{Proposition 4.4}, respectively, and that the lax double functor $F^\prime$ is symmetric lax monoidal by Proposition \ref{Proposition 4.5}. As $\bold{1}$ is trivially a symmetric monoidal pseudo double category and the functor $E$ is trivially a symmetric monoidal oplax double functor, it follows by definition that the pseudo comma double category $(E/F^\prime)$ is also symmetric monoidal. There are a fair number of commuting diagrams to check, many of which use the globular morphism $\fx\maps (M_1\ten N_1)\odot (M_2\ten N_2)\too[\sim] (M_1\odot M_2)\ten (N_1\odot N_2)$ between horizontal 1-cells which we will show how to obtain.

Let $M_{1},M_{2},N_{1}$ and $N_{2}$ be horizontal 1-cells given by decorated cospans
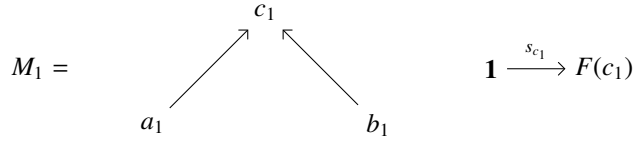
\[
\begin{tikzpicture}[scale=1.5]
\node (C') at (-1,0.5) {$M_{1}=$};
\node (A) at (0,0) {$a_{1}$};
\node (B) at (1,1) {$c_{1}$};
\node (C) at (2,0) {$b_{1}$};
\node (A') at (3,0.5) {$\bold{1}$};
\node (B') at (4,0.5) {$F(c_{1})$};
\path[->,font=\scriptsize,>=angle 90]
(A) edge node[above]{$$} (B)
(A') edge node[above]{$s_{c_{1}}$} (B')
(C) edge node[above]{$$} (B);
\end{tikzpicture}
\]
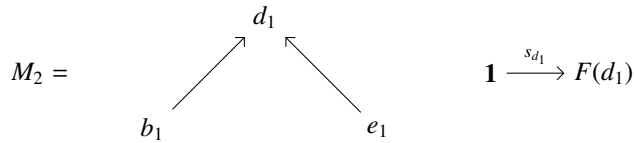
\[
\begin{tikzpicture}[scale=1.5]
\node (C') at (-1,0.5) {$M_{2}=$};
\node (A) at (0,0) {$b_{1}$};
\node (B) at (1,1) {$d_{1}$};
\node (C) at (2,0) {$e_{1}$};
\node (A') at (3,0.5) {$\bold{1}$};
\node (B') at (4,0.5) {$F(d_{1})$};
\path[->,font=\scriptsize,>=angle 90]
(A) edge node[above]{$$} (B)
(A') edge node[above]{$s_{d_{1}}$} (B')
(C) edge node[above]{$$} (B);
\end{tikzpicture}
\]
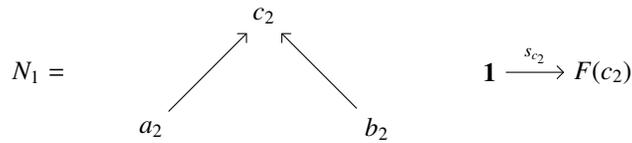
\[
\begin{tikzpicture}[scale=1.5]
\node (C') at (-1,0.5) {$N_{1}=$};
\node (A) at (0,0) {$a_{2}$};
\node (B) at (1,1) {$c_{2}$};
\node (C) at (2,0) {$b_{2}$};
\node (A') at (3,0.5) {$\bold{1}$};
\node (B') at (4,0.5) {$F(c_{2})$};
\path[->,font=\scriptsize,>=angle 90]
(A) edge node[above]{$$} (B)
(A') edge node[above]{$s_{c_{2}}$} (B')
(C) edge node[above]{$$} (B);
\end{tikzpicture}
\]
and
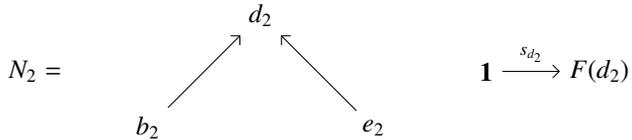
\[
\begin{tikzpicture}[scale=1.5]
\node (C') at (-1,0.5) {$N_{2}=$};
\node (A) at (0,0) {$b_{2}$};
\node (B) at (1,1) {$d_{2}$};
\node (C) at (2,0) {$e_{2}$};
\node (A') at (3,0.5) {$\bold{1}$};
\node (B') at (4,0.5) {$F(d_{2})$};
\path[->,font=\scriptsize,>=angle 90]
(A) edge node[above]{$$} (B)
(A') edge node[above]{$s_{d_{2}}$} (B')
(C) edge node[above]{$$} (B);
\end{tikzpicture}
.\]
Then we have that $M_{1} \otimes N_{1},M_{2} \otimes N_{2},M_{1} \odot M_{2}$ and $N_{1} \odot N_{2}$ are given by
\[
\begin{tikzpicture}[scale=1.5]
\node (C') at (-1,0.5) {$M_{1} \otimes N_{1}=$};
\node (A) at (0,0) {$a_{1}+a_{2}$};
\node (B) at (1,1) {$c_{1}+c_{2}$};
\node (C) at (2,0) {$b_{1}+b_{2}$};
\node (A') at (-2,-1) {$\bold{1}$};
\node (B') at (-1,-1) {$\bold{1} \times \bold{1}$};
\node (D) at (1,-1) {$F(c_{1}) \times F(c_{2})$};
\node (E) at (3,-1) {$F(c_{1}+c_{2})$};
\path[->,font=\scriptsize,>=angle 90]
(A) edge node[above]{$$} (B)
(A') edge node[above]{$\ell^{-1}$} (B')
(B') edge node[above]{$s_{c_{1}} \times s_{c_{2}}$} (D)
(D) edge node[above]{$\phi_{c_{1},c_{2}}$} (E)
(C) edge node[above]{$$} (B);
\end{tikzpicture}
\]
\[
\begin{tikzpicture}[scale=1.5]
\node (C') at (-1,0.5) {$M_{2} \otimes N_{2}=$};
\node (A) at (0,0) {$b_{1}+b_{2}$};
\node (B) at (1,1) {$d_{1}+d_{2}$};
\node (C) at (2,0) {$e_{1}+e_{2}$};
\node (A') at (-2,-1) {$\bold{1}$};
\node (B') at (-1,-1) {$\bold{1} \times \bold{1}$};
\node (D) at (1,-1) {$F(d_{1}) \times F(d_{2})$};
\node (E) at (3,-1) {$F(d_{1}+d_{2})$};
\path[->,font=\scriptsize,>=angle 90]
(A) edge node[above]{$$} (B)
(A') edge node[above]{$\ell^{-1}$} (B')
(B') edge node[above]{$s_{d_{1}} \times s_{d_{2}}$} (D)
(D) edge node[above]{$\phi_{d_{1},d_{2}}$} (E)
(C) edge node[above]{$$} (B);
\end{tikzpicture}
\]
\[
\begin{tikzpicture}[scale=1.5]
\node (C') at (-1,0.5) {$M_{1} \odot M_{2}=$};
\node (A) at (0,0) {$a_{1}$};
\node (B) at (1,1) {$c_{1}+_{b_{1}}d_{1}$};
\node (C) at (2,0) {$e_{1}$};
\node (A') at (-2.5,-1) {$\bold{1}$};
\node (B') at (-1.5,-1) {$\bold{1} \times \bold{1}$};
\node (D) at (0.5,-1) {$F(c_{1}) \times F(d_{1})$};
\node (E) at (2.5,-1) {$F(c_{1}+d_{1})$};
\node (F) at (4.5,-1) {$F(c_{1}+_{b_{1}}d_{1})$};
\path[->,font=\scriptsize,>=angle 90]
(A) edge node[above]{$$} (B)
(A') edge node[above]{$\ell^{-1}$} (B')
(B') edge node[above]{$s_{c_{1}} \times s_{d_{1}}$} (D)
(D) edge node[above]{$\phi_{c_{1},d_{1}}$} (E)
(E) edge node[above]{$F(j_{c_{1},d_{1}})$} (F)
(C) edge node[above]{$$} (B);
\end{tikzpicture}
\]
and
\[
\begin{tikzpicture}[scale=1.5]
\node (C') at (-1,0.5) {$N_{1} \odot N_{2}=$};
\node (A) at (0,0) {$a_{2}$};
\node (B) at (1,1) {$c_{2}+_{b_{2}}d_{2}$};
\node (C) at (2,0) {$e_{2}$};
\node (A') at (-2.5,-1) {$\bold{1}$};
\node (B') at (-1.5,-1) {$\bold{1} \times \bold{1}$};
\node (D) at (0.5,-1) {$F(c_{2}) \times F(d_{2})$};
\node (E) at (2.5,-1) {$F(c_{2}+d_{2})$};
\node (F) at (4.5,-1) {$F(c_{2}+_{b_{2}}d_{2})$};
\path[->,font=\scriptsize,>=angle 90]
(A) edge node[above]{$$} (B)
(A') edge node[above]{$\ell^{-1}$} (B')
(B') edge node[above]{$s_{c_{2}} \times s_{d_{2}}$} (D)
(D) edge node[above]{$\phi_{c_{2},d_{2}}$} (E)
(E) edge node[above]{$F(j_{c_{2},d_{2}})$} (F)
(C) edge node[above]{$$} (B);
\end{tikzpicture}
\]
so then $(M_{1} \otimes N_{1}) \odot (M_{2} \otimes N_{2})$ and $(M_{1} \odot M_{2}) \otimes (N_{1} \odot N_{2})$ are given by
\[
\begin{tikzpicture}[scale=1.5]
\node (C') at (-2,0.5) {$(M_{1} \otimes N_{1}) \odot (M_{2} \otimes N_{2})=$};
\node (A) at (0,0) {$a_{1}+a_{2}$};
\node (B) at (1,1) {$(c_{1}+c_{2})+_{b_{1}+b_{2}}(d_{1}+d_{2})$};
\node (C) at (2,0) {$e_{1}+e_{2}$};
\node (A') at (-2.5,-1) {$\bold{1}$};
\node (D) at (-0.5,-1) {$F(c_{1}+c_{2}) \times F(d_{1}+d_{2})$};
\node (E) at (2.5,-1) {$F((c_{1}+c_{2})+(d_{1}+d_{2}))$};
\node (F) at (0.5,-2) {$F((c_{1}+c_{2})+_{b_{1}+b_{2}}(d_{1}+d_{2}))$};
\path[->,font=\scriptsize,>=angle 90]
(A) edge node[above]{$$} (B)
(A') edge node[above]{$\psi$} (D)
(D) edge node[above]{$\phi$} (E)
(E) edge node[below]{$F(j)$} (F)
(C) edge node[above]{$$} (B);
\end{tikzpicture}
\]
where $$\psi=(\phi_{c_{1},c_{2}} \circ (s_{c_{1}} \times s_{c_{2}}) \circ \ell^{-1}) \times (\phi_{d_{1},d_{2}} \circ (s_{d_{1}} \times s_{d_{2}}) \circ \ell^{-1})$$ is the product of the maps given above, and $$F(j)=F(j_{c_{1}+c_{2},d_{1}+d_{2}})$$ and
\[
\begin{tikzpicture}[scale=1.5]
\node (C') at (-2,0.5) {$(M_{1} \odot M_{2}) \otimes (N_{1} \odot N_{2})=$};
\node (A) at (0,0) {$a_{1}+a_{2}$};
\node (B) at (1,1) {$(c_{1}+_{b_{1}}d_{1})+(c_{2}+_{b_{2}}d_{2})$};
\node (C) at (2,0) {$e_{1}+e_{2}$};
\node (A') at (-3,-1) {$\bold{1}$};
\node (D) at (-1,-1) {$F(c_{1}+_{b_{1}}d_{1}) \times F(c_{2}+_{b_{2}}d_{2})$};
\node (E) at (3,-1) {$F((c_{1}+_{b_{1}}d_{1})+(c_{2}+_{b_{2}}d_{2}))$};
\path[->,font=\scriptsize,>=angle 90]
(A) edge node[above]{$$} (B)
(A') edge node[above]{$\theta$} (D)
(D) edge node[above]{$\phi_{c_{1}+_{b_{1}}d_{1},c_{2}+_{b_{2}}d_{2}}$} (E)
(C) edge node[above]{$$} (B);
\end{tikzpicture}
\]
where $$\theta = (F(j_{c_{1},d_{1}}) \circ \phi_{c_{1},d_{1}} \circ (s_{c_{1}} \times s_{d_{1}}) \circ \ell^{-1}) \times (F(j_{c_{2},d_{2}}) \circ \phi_{c_{2},d_{2}} \circ (s_{c_{2}} \times s_{d_{2}}) \circ \ell^{-1}).$$
The globular morphism 
 \[\fx\maps (M_1\ten N_1)\odot (M_2\ten N_2)\too[\sim] (M_1\odot M_2)\ten (N_1\odot N_2)\]
is then given by the universal map $\alpha \colon (c_{1}+c_{2})+_{b_{1}+b_{2}}(d_{1}+d_{2}) \to (c_{1}+_{b_{1}}d_{1})+(c_{2}+_{b_{2}}d_{2})$ given by Proposition \ref{Proposition 4.2} that makes
\[
\begin{tikzpicture}[scale=1.5]
\node (A) at (1,1) {$a_{1}+a_{2}$};
\node (B) at (3,1) {$(c_{1}+c_{2})+_{b_{1}+b_{2}}(d_{1}+d_{2})$};
\node (C) at (5,1) {$e_{1}+e_{2}$};
\node (A') at (1,0) {$a_{1}+a_{2}$};
\node (B') at (3,0) {$(c_{1}+_{b_{1}}d_{1})+(c_{2}+_{b_{2}}d_{2})$};
\node (C') at (5,0) {$e_{1}+e_{2}$};
\path[->,font=\scriptsize,>=angle 90]
(A) edge node[above]{$$} (B)
(A') edge node[above]{$$} (B')
(A) edge node[left]{$$} (A')
(C) edge node[above]{$$} (B)
(C') edge node[above]{$$} (B')
(B) edge node[left]{$\alpha$} (B')
(C) edge node[right]{$$} (C');
\end{tikzpicture}
\] 
commute. This universal map $\alpha$ also gives us the map $F(\alpha)$ which makes
\[
\begin{tikzpicture}[scale=1.5]
\node (D) at (2,0.5) {$\bold{1}$};
\node (E) at (5,0) {$F((c_{1}+c_{2})+_{b_{1}+b_{2}}(d_{1}+d_{2}))$};
\node (F) at (5,1) {$F((c_{1}+_{b_{1}}d_{1})+(c_{2}+_{b_{2}}d_{2}))$};
\path[->,font=\scriptsize,>=angle 90]
(D) edge node[below]{$$} (E)
(D) edge node[above]{$$} (F)
(F) edge node[right]{$F(\alpha)$} (E);
\end{tikzpicture}
\]
commute. The remaining details are similar to those given here.
\end{proof}


A more sophisticated method of proof uses the theory of 2-monads. Let $\textbf{Graph}(\textbf{Cat})$ be the 2-category of graphs internal to $\textbf{Cat}$, graph morphisms internal to $\textbf{Cat}$, and transformations between these. There is a 2-monad $T$ on $\textbf{Graph}(\textbf{Cat})$ whose strict algebras are pseudo double categories. The strict (resp.\ pseudo, lax) morphisms between these algebras are strict (resp.\ pseudo, lax) double functors. There is thus a 2-category $T\textbf{Alg}_\ell$ consisting of pseudo double categories, lax double functors and transformations. As the oplax double functor $E \colon \bold{1} \to \lB \lD$ constructed in Theorem \ref{Theorem 4.8} is in fact \emph{strict}, a result of Lack \cite[Prop.\ 4.6]{Lack} implies that $(E/F^\prime)$ exists as a comma object in $T\textbf{Alg}_\ell$.

\begin{prop}\label{Proposition 4.9}
The symmetric monoidal double category $(E/F^\prime)$ is fibrant.
\end{prop}

\begin{proof}
We have that $\text{Cospan}(\textbf{C})$ is fibrant; a companion of a horizointal 1-cell $f \colon A \to B$ is the cospan $f \colon A \to B \leftarrow B \colon \id_{B}$ with corresponding conjoint $\id_{B} \colon B \to B \leftarrow A \colon f$ and where $U_{A}$ is the identity cospan $\id_{A} \colon A \to A \leftarrow A \colon \id_{A}$ on the object $A$. It then follows that $(E/F^\prime)$ is also fibrant by choosing the trivial decoration for all of the above cospans which will then satisfy the required equations in the definition of fibrant, as these equations simply become the equations required to be satisfied for $\text{Cospan}(\textbf{C})$ to be fibrant.
\end{proof}



\section{Applications}\label{Section 5}

In this last section, we present two applications of the main result. The first involves the symmetric monoidal categories studied by Rosebrugh, Sabadini and Walters \cite{RSW}, in which a morphism is a directed graph with labeled edges and specified input and output nodes.  We can promote these categories to symmetric monoidal bicategories.  First, we make the following definitions:

\begin{defn}
A \textbf{graph} is a finite set $E$ of \textbf{edges} and a finite set $N$ \textbf{nodes} equipped with a pair of functions $s,t \colon E \to N$ that assign to each edge $e \in E$ its source and target, $s(e)$ and $t(e)$, respectively. In this case, we say that $e \in E$ is an edge from $s(e)$ to $t(e)$.
\end{defn}

\begin{defn}
Given a set $L$ of \textbf{labels}, an \define{$L$-graph} is a graph equipped with a function $\ell \colon E \to L$ which assigns a label to each edge.
\end{defn}

For example, if we take $L=(0,\infty)$, an $L$-graph is just a weighted graph, as discussed in the Introduction:

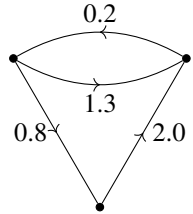
\begin{center}
  \begin{tikzpicture}[auto,scale=2.3]
    \node[circle,draw,inner sep=1pt,fill]         (A) at (0,0) {};
    \node[circle,draw,inner sep=1pt,fill]         (B) at (1,0) {};
    \node[circle,draw,inner sep=1pt,fill]         (C) at (0.5,-.86) {};
    \path (B) edge  [bend right,->-] node[above] {0.2} (A);
    \path (A) edge  [bend right,->-] node[below] {1.3} (B);
    \path (A) edge  [->-] node[left] {0.8} (C);
    \path (C) edge  [->-] node[right] {2.0} (B);
  \end{tikzpicture}
\end{center}
To turn a graph like this into a morphism in a category, we specify input and output nodes using a cospan of finite sets.

\begin{defn}
Given a set $L$ and finite sets $X$ and $Y$, an \define{$L$-circuit from $X$ to $Y$} is a cospan of finite sets
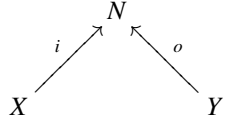
\[
 \begin{aligned}
      \xymatrix{
	& N \\  
	X \ar[ur]^{i} && Y \ar[ul]_{o}
      }
    \end{aligned}
\]
together with an $L$-graph

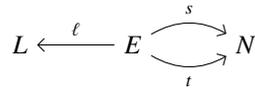
\[
\begin{tikzpicture}[scale=1.5]
\node (A) at (1,1) {$E$};
\node (B) at (2,1) {$N$};
\node (C) at (0,1) {$L$};
\path[->,font=\scriptsize,>=angle 90]
(A) edge node[above]{$\ell$} (C)
(A) edge[bend left] node[above]{$s$} (B)
(A) edge[bend right] node[below]{$t$} (B);
\end{tikzpicture}
\]
We call the sets $i(X), o(Y)$ and $\del N = i(X) \cup o(Y)$ the \textbf{inputs}, \textbf{outputs} and \textbf{terminals} of the $L$-circuit, respectively.
\end{defn}
We saw an example of an $L$-circuit in the introduction:
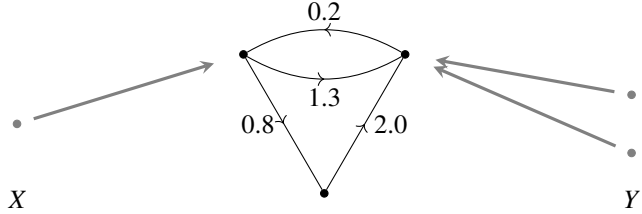
\begin{center}
  \begin{tikzpicture}[auto,scale=2.15]
    \node[circle,draw,inner sep=1pt,fill=gray,color=gray]         (x) at (-1.4,-.43) {};
    \node at (-1.4,-.9) {$X$};
    \node[circle,draw,inner sep=1pt,fill]         (A) at (0,0) {};
    \node[circle,draw,inner sep=1pt,fill]         (B) at (1,0) {};
    \node[circle,draw,inner sep=1pt,fill]         (C) at (0.5,-.86) {};
    \node[circle,draw,inner sep=1pt,fill=gray,color=gray]         (y1) at (2.4,-.25) {};
    \node[circle,draw,inner sep=1pt,fill=gray,color=gray]         (y2) at (2.4,-.61) {};
    \node at (2.4,-.9) {$Y$};
    \path (B) edge  [bend right,->-] node[above] {0.2} (A);
    \path (A) edge  [bend right,->-] node[below] {1.3} (B);
    \path (A) edge  [->-] node[left] {0.8} (C);
    \path (C) edge  [->-] node[right] {2.0} (B);
    \path[color=gray, very thick, shorten >=10pt, shorten <=5pt, ->, >=stealth] (x) edge (A);
    \path[color=gray, very thick, shorten >=10pt, shorten <=5pt, ->, >=stealth] (y1) edge (B);
    \path[color=gray, very thick, shorten >=10pt, shorten <=5pt, ->, >=stealth] (y2) edge (B);
  \end{tikzpicture}
\end{center}
Given another $L$-ciruit whose inputs match up with the outputs of the above $L$-circuit:
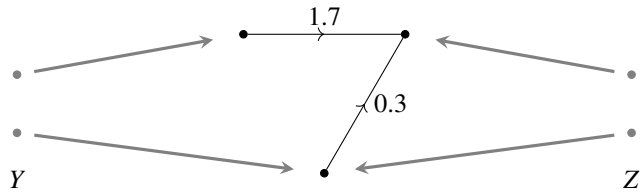
\begin{center}
  \begin{tikzpicture}[auto,scale=2.15]
    \node[circle,draw,inner sep=1pt,fill=gray,color=gray]         (y1) at (-1.4,-.25) {};
    \node[circle,draw,inner sep=1pt,fill=gray,color=gray]         (y2) at (-1.4,-.61) {};
    \node at (-1.4,-.9) {$Y$};
    \node[circle,draw,inner sep=1pt,fill]         (A) at (0,0) {};
    \node[circle,draw,inner sep=1pt,fill]         (B) at (1,0) {};
    \node[circle,draw,inner sep=1pt,fill]         (C) at (0.5,-.86) {};
    \node[circle,draw,inner sep=1pt,fill=gray,color=gray]         (z1) at (2.4,-.25) {};
    \node[circle,draw,inner sep=1pt,fill=gray,color=gray]         (z2) at (2.4,-.61) {};
    \node at (2.4,-.9) {$Z$};
    \path (A) edge  [->-] node[above] {1.7} (B);
    \path (C) edge  [->-] node[right] {0.3} (B);
    \path[color=gray, very thick, shorten >=10pt, shorten <=5pt, ->, >=stealth] (y1) edge (A);
    \path[color=gray, very thick, shorten >=10pt, shorten <=5pt, ->, >=stealth] (y2)
    edge (C);
    \path[color=gray, very thick, shorten >=10pt, shorten <=5pt, ->, >=stealth] (z1) edge (B);
    \path[color=gray, very thick, shorten >=10pt, shorten <=5pt, ->, >=stealth] (z2) edge (C);
  \end{tikzpicture}
\end{center}
we can compose them by making the following identifications
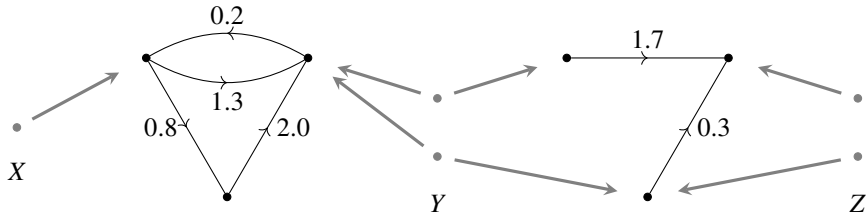
\begin{center}
  \begin{tikzpicture}[auto,scale=2.15]
    \node[circle,draw,inner sep=1pt,fill=gray,color=gray]         (x) at (-0.8,-.43) {};
    \node at (-0.8,-.7) {$X$};
    \node[circle,draw,inner sep=1pt,fill]         (A) at (0,0) {};
    \node[circle,draw,inner sep=1pt,fill]         (B) at (1,0) {};
    \node[circle,draw,inner sep=1pt,fill]         (C) at (0.5,-.86) {};
    \node[circle,draw,inner sep=1pt,fill=gray,color=gray]         (y1) at (1.8,-.25) {};
    \node[circle,draw,inner sep=1pt,fill=gray,color=gray]         (y2) at (1.8,-.61) {};
    \node at (1.8,-.9) {$Y$};
    \path (B) edge  [bend right,->-] node[above] {0.2} (A);
    \path (A) edge  [bend right,->-] node[below] {1.3} (B);
    \path (A) edge  [->-] node[left] {0.8} (C);
    \path (C) edge  [->-] node[right] {2.0} (B);
    \path[color=gray, very thick, shorten >=10pt, shorten <=5pt, ->, >=stealth] (x) edge (A);
    \path[color=gray, very thick, shorten >=10pt, shorten <=5pt, ->, >=stealth] (y1) edge (B);
    \path[color=gray, very thick, shorten >=10pt, shorten <=5pt, ->, >=stealth] (y2) edge (B);
    \node[circle,draw,inner sep=1pt,fill]         (A') at (2.6,0) {};
    \node[circle,draw,inner sep=1pt,fill]         (B') at (3.6,0) {};
    \node[circle,draw,inner sep=1pt,fill]         (C') at (3.1,-.86) {};
    \node[circle,draw,inner sep=1pt,fill=gray,color=gray]         (z1) at (4.4,-.25) {};
    \node[circle,draw,inner sep=1pt,fill=gray,color=gray]         (z2) at (4.4,-.61) {};
    \node at (4.4,-.9) {$Z$};
    \path (A') edge  [->-] node[above] {1.7} (B');
    \path (C') edge  [->-] node[right] {0.3} (B');
    \path[color=gray, very thick, shorten >=10pt, shorten <=5pt, ->, >=stealth] (y1) edge (A');
    \path[color=gray, very thick, shorten >=10pt, shorten <=5pt, ->, >=stealth] (y2)
    edge (C');
    \path[color=gray, very thick, shorten >=10pt, shorten <=5pt, ->, >=stealth] (z1) edge (B');
    \path[color=gray, very thick, shorten >=10pt, shorten <=5pt, ->, >=stealth] (z2) edge (C');
  \end{tikzpicture}
\end{center}
and obtain the following $L$-circuit:
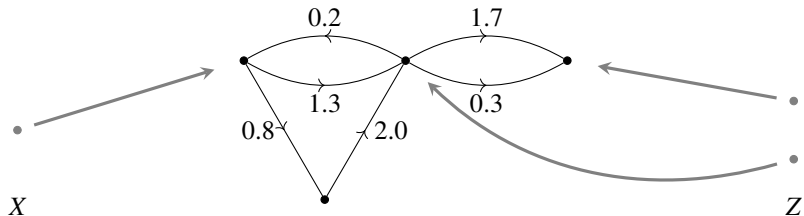
\begin{center}
  \begin{tikzpicture}[auto,scale=2.15]
    \node[circle,draw,inner sep=1pt,fill=gray,color=gray]         (x) at (-1.4,-.43) {};
    \node at (-1.4,-.9) {$X$};
    \node[circle,draw,inner sep=1pt,fill]         (A) at (0,0) {};
    \node[circle,draw,inner sep=1pt,fill]         (B) at (1,0) {};
    \node[circle,draw,inner sep=1pt,fill]         (C) at (0.5,-.86) {};
    \node[circle,draw,inner sep=1pt,fill]         (D) at (2,0) {};
    \node[circle,draw,inner sep=1pt,fill=gray,color=gray]         (z1) at (3.4,-.25) {};
    \node[circle,draw,inner sep=1pt,fill=gray,color=gray]         (z2) at (3.4,-.61) {};
    \node at (3.4,-.9) {$Z$};
    \path (B) edge  [bend right,->-] node[above] {0.2} (A);
    \path (A) edge  [bend right,->-] node[below] {1.3} (B);
    \path (A) edge  [->-] node[left] {0.8} (C);
    \path (C) edge  [->-] node[right] {2.0} (B);
    \path (B) edge  [bend left,->-] node[above] {1.7} (D);
    \path (B) edge  [bend right,->-] node[below] {0.3} (D);
    \path[color=gray, very thick, shorten >=10pt, shorten <=5pt, ->, >=stealth] (x) edge (A);
    \path[color=gray, very thick, shorten >=10pt, shorten <=5pt, ->, >=stealth] (z1)
    edge (D);
    \path[bend left, color=gray, very thick, shorten >=10pt, shorten <=5pt, ->, >=stealth] (z2)
    edge (B);
  \end{tikzpicture}
\end{center}
Following the result of the paper's main theorem, a 2-morphism will be a globular 2-morphism between cospans of finite sets whose apices are decorated with $L$-graphs. This amounts to a map $h \colon N \to N^\prime$ between the apices such that the decorations of the $L$-graphs are preserved.
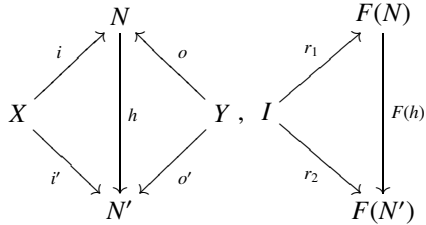
\[
\begin{aligned}
      \xymatrix{
	& N \ar[dd]^{h} \\  
	X \ar[ur]^{i} \ar[dr]_{i^\prime} && Y \ar[ul]_{o} \ar[dl]^{o^\prime}\\
          & N'
      }
    \end{aligned},
\begin{aligned}
      \xymatrix{
	& F(N) \ar[dd]^{F(h)} \\
	I \ar[ur]^{r_{1}} \ar[dr]_{r_{2}} \\
           & F(N^\prime)
      }
    \end{aligned}
\]
If we have an $L$-graph
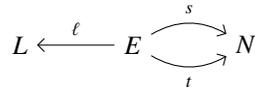
\[
\begin{tikzpicture}[scale=1.5]
\node (A) at (1,1) {$E$};
\node (B) at (2,1) {$N$};
\node (C) at (0,1) {$L$};
\path[->,font=\scriptsize,>=angle 90]
(A) edge node[above]{$\ell$} (C)
(A) edge[bend left] node[above]{$s$} (B)
(A) edge[bend right] node[below]{$t$} (B);
\end{tikzpicture}
\]
which decorates the set $N$, the $L$-graph
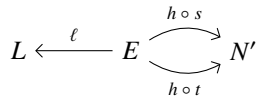
\[
\begin{tikzpicture}[scale=1.5]
\node (A) at (1,1) {$E$};
\node (B) at (2,1) {$N^\prime$};
\node (C) at (0,1) {$L$};
\path[->,font=\scriptsize,>=angle 90]
(A) edge node[above]{$\ell$} (C)
(A) edge[bend left] node[above]{$h \circ s$} (B)
(A) edge[bend right] node[below]{$h \circ t$} (B);
\end{tikzpicture}
\]
decorates the set $N^\prime$ and makes the diagram on the right above commute. We can also tensor two $L$-circuits by formally placing them side by side; for example, if we tensor 

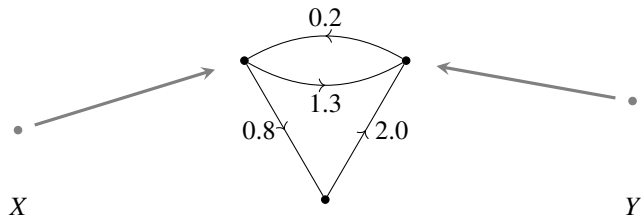
\begin{center}
  \begin{tikzpicture}[auto,scale=2.15]
    \node[circle,draw,inner sep=1pt,fill=gray,color=gray]         (x) at (-1.4,-.43) {};
    \node at (-1.4,-.9) {$X$};
    \node[circle,draw,inner sep=1pt,fill]         (A) at (0,0) {};
    \node[circle,draw,inner sep=1pt,fill]         (B) at (1,0) {};
    \node[circle,draw,inner sep=1pt,fill]         (C) at (0.5,-.86) {};
    \node[circle,draw,inner sep=1pt,fill=gray,color=gray]         (y1) at (2.4,-.25) {};
    \node at (2.4,-.9) {$Y$};
    \path (B) edge  [bend right,->-] node[above] {0.2} (A);
    \path (A) edge  [bend right,->-] node[below] {1.3} (B);
    \path (A) edge  [->-] node[left] {0.8} (C);
    \path (C) edge  [->-] node[right] {2.0} (B);
    \path[color=gray, very thick, shorten >=10pt, shorten <=5pt, ->, >=stealth] (x) edge (A);
    \path[color=gray, very thick, shorten >=10pt, shorten <=5pt, ->, >=stealth] (y1) edge (B);
  \end{tikzpicture}
\end{center}
with
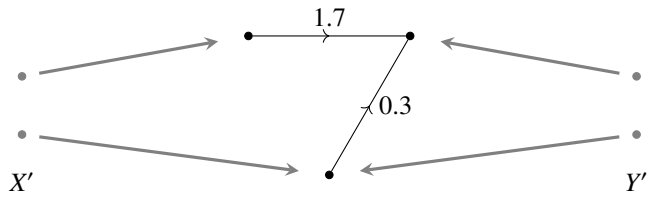
\begin{center}
  \begin{tikzpicture}[auto,scale=2.15]
    \node[circle,draw,inner sep=1pt,fill=gray,color=gray]         (y1) at (-1.4,-.25) {};
    \node[circle,draw,inner sep=1pt,fill=gray,color=gray]         (y2) at (-1.4,-.61) {};
    \node at (-1.4,-.9) {$X^\prime$};
    \node[circle,draw,inner sep=1pt,fill]         (A) at (0,0) {};
    \node[circle,draw,inner sep=1pt,fill]         (B) at (1,0) {};
    \node[circle,draw,inner sep=1pt,fill]         (C) at (0.5,-.86) {};
    \node[circle,draw,inner sep=1pt,fill=gray,color=gray]         (z1) at (2.4,-.25) {};
    \node[circle,draw,inner sep=1pt,fill=gray,color=gray]         (z2) at (2.4,-.61) {};
    \node at (2.4,-.9) {$Y^\prime$};
    \path (A) edge  [->-] node[above] {1.7} (B);
    \path (C) edge  [->-] node[right] {0.3} (B);
    \path[color=gray, very thick, shorten >=10pt, shorten <=5pt, ->, >=stealth] (y1) edge (A);
    \path[color=gray, very thick, shorten >=10pt, shorten <=5pt, ->, >=stealth] (y2)
    edge (C);
    \path[color=gray, very thick, shorten >=10pt, shorten <=5pt, ->, >=stealth] (z1) edge (B);
    \path[color=gray, very thick, shorten >=10pt, shorten <=5pt, ->, >=stealth] (z2) edge (C);
  \end{tikzpicture}
\end{center}
we get
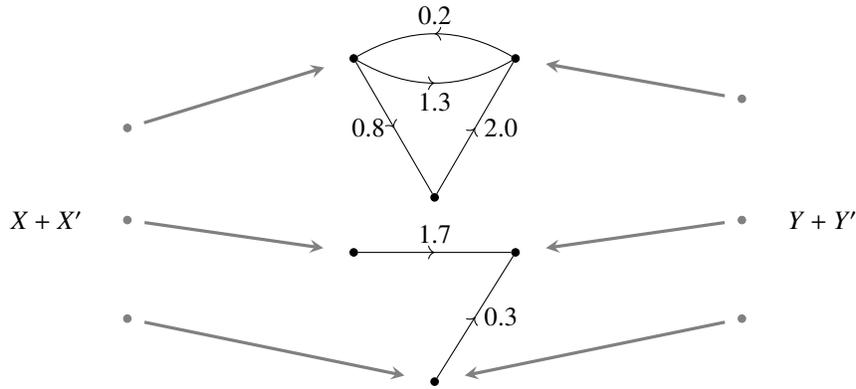
\begin{center}
  \begin{tikzpicture}[auto,scale=2.15]
    \node[circle,draw,inner sep=1pt,fill=gray,color=gray]         (x) at (-1.4,-.43) {};
    \node[circle,draw,inner sep=1pt,fill=gray,color=gray]         (w1) at (-1.4,-1) {};
    \node[circle,draw,inner sep=1pt,fill=gray,color=gray]         (w2) at (-1.4,-1.61) {};
    \node at (-1.9,-1) {$X+X^\prime$};
    \node[circle,draw,inner sep=1pt,fill]         (A) at (0,0) {};
    \node[circle,draw,inner sep=1pt,fill]         (B) at (1,0) {};
    \node[circle,draw,inner sep=1pt,fill]         (C) at (0.5,-.86) {};
    \node[circle,draw,inner sep=1pt,fill]         (D) at (0,-1.2) {};
    \node[circle,draw,inner sep=1pt,fill]         (E) at (1,-1.2) {};
    \node[circle,draw,inner sep=1pt,fill]         (F) at (0.5,-2) {};
    \node[circle,draw,inner sep=1pt,fill=gray,color=gray]         (y1) at (2.4,-.25) {};
    \node[circle,draw,inner sep=1pt,fill=gray,color=gray]         (z1) at (2.4,-1) {};
    \node[circle,draw,inner sep=1pt,fill=gray,color=gray]         (z2) at (2.4,-1.61) {};
    \node at (2.9,-1) {$Y+Y^\prime$};
    \path (D) edge  [->-] node[above] {1.7} (E);
    \path (F) edge  [->-] node[right] {0.3} (E);
    \path (B) edge  [bend right,->-] node[above] {0.2} (A);
    \path (A) edge  [bend right,->-] node[below] {1.3} (B);
    \path (A) edge  [->-] node[left] {0.8} (C);
    \path (C) edge  [->-] node[right] {2.0} (B);
    \path[color=gray, very thick, shorten >=10pt, shorten <=5pt, ->, >=stealth] (w1) edge (D);
    \path[color=gray, very thick, shorten >=10pt, shorten <=5pt, ->, >=stealth] (w2)
    edge (F);
    \path[color=gray, very thick, shorten >=10pt, shorten <=5pt, ->, >=stealth] (x) edge (A);
    \path[color=gray, very thick, shorten >=10pt, shorten <=5pt, ->, >=stealth] (y1) edge (B);
   \path[color=gray, very thick, shorten >=10pt, shorten <=5pt, ->, >=stealth] (z1) edge (E);
    \path[color=gray, very thick, shorten >=10pt, shorten <=5pt, ->, >=stealth] (z2) edge (F);
  \end{tikzpicture}
\end{center}
More formally, given two $L$-circuits 

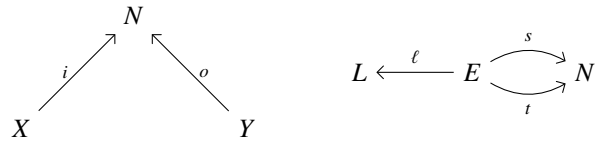
\[
 \begin{aligned}
\begin{tikzpicture}[scale=1.5]
\node (A) at (1,1) {$E$};
\node (B) at (2,1) {$N$};
\node (C) at (0,1) {$L$};
\node (D) at (-1,0.5) {$Y$};
\node (E) at (-2,1.5) {$N$};
\node (F) at (-3,0.5) {$X$};
\path[->,font=\scriptsize,>=angle 90]
(A) edge node[above]{$\ell$} (C)
(A) edge[bend left] node[above]{$s$} (B)
(A) edge[bend right] node[below]{$t$} (B)
(D) edge node [above,right] {$o$} (E)
(F) edge node [above,left]{$i$} (E);
\end{tikzpicture}
\end{aligned}
\]
and
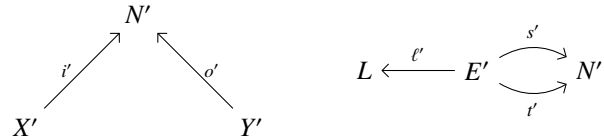
\[
 \begin{aligned}
\begin{tikzpicture}[scale=1.5]
\node (A) at (1,1) {$E^\prime$};
\node (B) at (2,1) {$N^\prime$};
\node (C) at (0,1) {$L$};
\node (D) at (-1,0.5) {$Y^\prime$};
\node (E) at (-2,1.5) {$N^\prime$};
\node (F) at (-3,0.5) {$X^\prime$};
\path[->,font=\scriptsize,>=angle 90]
(A) edge node[above]{$\ell^\prime$} (C)
(A) edge[bend left] node[above]{$s^\prime$} (B)
(A) edge[bend right] node[below]{$t^\prime$} (B)
(D) edge node [above,right] {$o^\prime$} (E)
(F) edge node [above,left]{$i^\prime$} (E);
\end{tikzpicture}
\end{aligned}
\]
their tensor product is
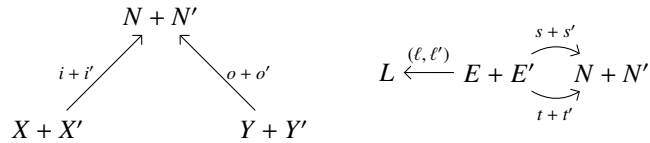
\[
 \begin{aligned}
\begin{tikzpicture}[scale=1.5]
\node (A) at (1,1) {$E+E^\prime$};
\node (B) at (2,1) {$N+N^\prime$};
\node (C) at (0,1) {$L$};
\node (D) at (-1,0.5) {$Y+Y^\prime$};
\node (E) at (-2,1.5) {$N+N^\prime$};
\node (F) at (-3,0.5) {$X+X^\prime$};
\path[->,font=\scriptsize,>=angle 90]
(A) edge node[above]{$(\ell,\ell^\prime)$} (C)
(A) edge[bend left] node[above]{$s+s^\prime$} (B)
(A) edge[bend right] node[below]{$t+t^\prime$} (B)
(D) edge node [above,right] {$o+o^\prime$} (E)
(F) edge node [above,left]{$i+i^\prime$} (E);
\end{tikzpicture}
\end{aligned}
\]




Rosebrugh, Sabadini and Walters \cite{RSW} constructed a symmetric monoidal category where objects are finite sets and morphisms are isomorphism classes of $L$-circuits (see also \cite{BCR,BF}). Theorem \ref{Theorem 2.2} lets us `categorify' their result, obtaining a symmetric monoidal bicategory where the morphisms are actual $L$-circuits.

\begin{thm}
For any set $L$, there is a symmetric monoidal bicategory \define{$L$-Circ} where the objects are finite sets, the 1-morphisms are $L$-circuits, with composition and the tensor product of 1-morphisms defined as above. 2-morphisms are maps between apices of two cospans with identical feet such that the following diagrams commute.
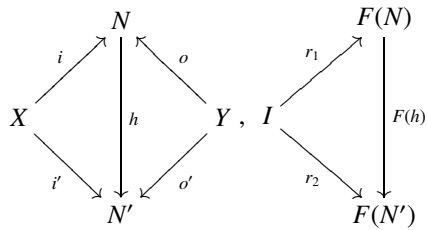
\[
\begin{aligned}
      \xymatrix{
	& N \ar[dd]^{h} \\  
	X \ar[ur]^{i} \ar[dr]_{i^\prime} && Y \ar[ul]_{o} \ar[dl]^{o^\prime}\\
          & N'
      }
    \end{aligned},
\begin{aligned}
      \xymatrix{
	& F(N) \ar[dd]^{F(h)} \\
	I \ar[ur]^{r_{1}} \ar[dr]_{r_{2}} \\
           & F(N^\prime)
      }
    \end{aligned}
\]
\end{thm}

\begin{proof}
We have a symmetric lax monoidal functor $F \colon \bold{FinSet} \to \bold{Set}$ that maps each finite set $X$ to $F(X)$, which is the set of all possible $L$-circuits on $X$. The functor $F$ is symmetric lax monoidal since we have maps $F(X) \times F(Y) \xrightarrow{\phi_{X,Y}} F(X+Y)$ that send a pair of $L$-circuits to the tensor product of the two $L$-circuits which gives rise to an $L$-circuit on $X+Y$, and $1 \to F(\emptyset)$ given by the empty $L$-circuit. We also have that $\bold{FinSet}$ is finitely cocomplete and $\bold{Set}$ is symmetric monoidal. The result follows from Thereom \ref{Theorem 2.2}.
\end{proof}

We can also obtain this theorem from the work of Stay \cite{Stay} if we treat \define{$L$-Circ} as a sub-bicategory of the symmetric monoidal bicategory of graphs.  His work even implies that \define{$L$-Circ} is a `compact' symmetric monoidal bicategory. 

For an example that cannot be handled using Stay's technique, we turn to the theory of dynamical systems. A dynamical system is a vector field, thought of as a system of first-order ordinary differential equations.  Chemical reaction networks give dynamical systems that are algebraic vector fields on $\mathbb{R}^n$: that is, vector fields with polynomial coefficients. In studying chemical reaction networks with inputs and outputs, Baez and Pollard \cite{BP} constructed a symmetric monoidal category where the morphisms are `open' dynamical systems.  We can promote this to a symmetric monoidal bicategory as follows.

We define a symmetric lax monoidal functor $D:\bold{FinSet} \to \bold{Set}$ as follows.  For any finite set $S$, let $D(S)$ be the set of all algebraic vector fields on $\mathbb{R}^S$:
$$D(S) = \{v_{S} \colon \mathbb{R}^{S} \to \mathbb{R}^{S} : \; v_S \textrm{ is algebraic} \}$$ 
For any function $f \colon S \to S^\prime$ between finite sets, define $D(f) \colon D(S) \to D(S^\prime)$ by 
$$D(f)(v_{S}) =f_{\ast}v_{S}f^\ast,$$ 
where $f^\ast \colon \mathbb{R}^{S^\prime} \to \mathbb{R}^{S}$ is the pullback defined by 
$$f^\ast (c_{S})(\sigma)=c_{S^\prime}(f(\sigma))$$ 
and $f_{\ast} \colon \mathbb{R}^S \to \mathbb{R}^{S^\prime}$ is the pushforward defined by 
$$f_{\ast}(c_{S^\prime})(\sigma^\prime) = \sum_{\{\sigma \in S : \; f(\sigma) = \sigma^\prime\}} c_{S}(\sigma)$$
where $c_{S} \in \mathbb{R}^{S}$ and $c_{S^\prime} \in \mathbb{R}^{S^\prime}$.  The functoriality of $D$ follows from the pushforward being covariant and the pullback being contravariant.  We make $D$ into a lax monoidal functor using the the unique map $\phi_{1} \colon 1 \to F(\emptyset)$ and the map $\phi_{S,S^\prime} \colon D(S) \times D(S^\prime) \to D(S+S^\prime)$ that sends a pair of vector fields $v_{S} \colon \mathbb{R}^{S} \to \mathbb{R}^{S}$ and $v_{S^\prime} \colon \mathbb{R}^{S^\prime} \to \mathbb{R}^{S^\prime}$ to $v_{S+S^\prime} \colon \mathbb{R}^{S+S^\prime} \to \mathbb{R}^{S+S^\prime}$ defined using the canonical isomorphism $\mathbb{R}^{S} \times \mathbb{R}^{S^\prime} \cong \mathbb{R}^{S+S^\prime}$. Furthermore, if we denote the braidings of $(\bold{FinSet},+,\emptyset)$ and $(\bold{Set},\times,1)$ by $\beta$, we have that the functor $D$ is symmetric as the following diagram commutes:

\[
\begin{tikzpicture}[scale=1.5]
\node (A) at (1,1) {$D(S) \times D(S^\prime)$};
\node (C) at (4,1) {$D(S^\prime) \times D(S)$};
\node (A') at (1,-1) {$D(S+S^\prime)$};
\node (C') at (4,-1) {$D(S^\prime + S)$};
\path[->,font=\scriptsize,>=angle 90]
(A) edge node[left]{$\phi_{S,S^\prime}$} (A')
(C) edge node[right]{$\phi_{S^\prime,S}$} (C')
(A) edge node[above]{$\beta_{D(S),D(S^\prime)}$} (C)
(A') edge node[above]{$D(\beta_{S,S^\prime})$} (C');
\end{tikzpicture}
\]
We have the following result due to Baez and Pollard \cite{BP}.


\begin{thm}
There is a decorated cospan category {\bf VectField} where an object is a finite set and a morphism is an isomorphism class of cospans of finite sets decorated by vector fields.
\end{thm}


We can categorify the above theorem by taking a 2-morphism between decorated cospans to be a map $h \maps S \to T$ between their apices making the usual diagrams commute:

\[
\begin{aligned}
      \xymatrix{
	& S \ar[dd]^{h} \\  
	X \ar[ur]^{i} \ar[dr]_{i^\prime} && Y \ar[ul]_{o} \ar[dl]^{o^\prime}\\
          & T
      }
    \end{aligned},
\begin{aligned}
      \xymatrix{
	& D(S) \ar[dd]^{D(h)} \\
	I \ar[ur]^{r_{1}} \ar[dr]_{r_{2}} \\
           & D(T)
      }
    \end{aligned}
\]
As $h:S \to T$ is a function, $D(h)$ induces a vector field on $\mathbb{R}^T$ given a vector field on $\mathbb{R}^S$ as prescribed above.

\begin{thm}
There is a symmetric monoidal bicategory {\bf VectField} where an object is a finite set, a morphism is a cospan of finite sets decorated by a vector field and a 2-morphism is a map of finite sets such that the above two diagrams commute.
\end{thm}

\begin{proof}
We apply Theorem \ref{Theorem 2.2} to the symmetric lax monoidal functor $D \colon \bold{FinSet} \to \bold{Set}$ as previously described.
\end{proof}

\section{Acknowledgements}

I would like to thank John Baez for teaching me everything that I know about category theory, and not only. No one has had a more profound influence on my mathematical career. I would also like to thank Michael Shulman, whose idea it was to approach this problem from the perspective of double categories and comma objects, as well as how to go about it and for an incredibly helpful correspondence via email.

\end{document}